\documentclass[11pt]{article}
\usepackage[T1]{fontenc}
\usepackage{amsmath,amssymb,amsthm,mathrsfs,times,color}
\usepackage[margin=1in]{geometry}
\usepackage{microtype}
\usepackage[colorlinks,linkcolor=blue,anchorcolor=blue,citecolor=blue]{hyperref}
\allowdisplaybreaks[3]
\numberwithin{equation}{section}
\newtheorem{prop}{Proposition}[section]
\newtheorem{theo}[prop]{Theorem}
\newtheorem{lem}[prop]{Lemma}
\newtheorem{cor}[prop]{Corollary}
\theoremstyle{remark}
\newtheorem*{rem}{Remark}
\theoremstyle{plain}
\newcommand{\beq}{\begin{equation}}
\newcommand{\eeq}{\end{equation}}
\newcommand{\beqn}{\begin{eqnarray}}
\newcommand{\eeqn}{\end{eqnarray}}
\newcommand{\beqs}{\begin{eqnarray*}}
\newcommand{\eeqs}{\end{eqnarray*}}
\def\no{\nonumber}
\def\pbp{\sqrt{-1}\partial\bar\partial}
\def\Na{\nabla}
\def\tr{{\rm tr}}
\newcommand{\PSH}{\operatorname{PSH}}
\newcommand{\Dpar}{\mathscr D}
\title{Global $C^{2,\alpha}$ estimate
for complex Monge--Amp\`ere equations\\ on compact K\"ahler manifolds}
\author{Haozhao Li\footnote{Supported by NSFC grant No. 12471058 and a research fund from the University of Science and Technology of China.}, \quad Bing Wang\footnote{Supported by NSFC grant No. 12431003 and a research fund from the University of Science and Technology of China.} \quad and \quad Kai Zheng
\footnote{Supported by NSFC grant No. 12571091.}}
\date{}

\begin{document}
\maketitle

\begin{abstract}
In this paper, we prove a global $C^{2,\alpha}$ estimate for the complex Monge--Amp\`ere equation with a positive $C^\alpha$ right-hand side on a compact K\"ahler manifold.
\end{abstract}
\tableofcontents

\section{Introduction}
\label{sec:introduction}

Let $(M^n,g)$ be a connected compact K\"ahler manifold of complex dimension $n\geq2$ with the
corresponding K\"ahler form $\omega_g$. Consider the equation
\beq
 (\omega_g+\pbp\varphi)^n=e^F\omega_g^n.
 \label{eq:intro-MA}
\eeq
Here $\omega_\varphi=\omega_g+\pbp\varphi$ and we call  $e^F$ the density with respect to $\omega_g^n$.
We write $d\mu_g=V_g^{-1}\omega_g^n$, where
$V_g=\int_M\omega_g^n$, so the compatibility condition is $\int_Me^F\,d\mu_g=1$. For smooth $F$
satisfying this condition, Yau \cite{Yau} proved the existence of a smooth solution, thereby
solving the Calabi conjecture. The existence of K\"ahler--Einstein metrics in the case of
negative first Chern class was proved independently by Aubin \cite{Au} and Yau \cite{Yau}.
In this paper, we prove a global $C^{2,\alpha}$ estimate for \eqref{eq:intro-MA} when
$F\in C^\alpha$.

The classical background includes the linear elliptic estimates of Schauder \cite{Schauder1934}
and the work of Nirenberg \cite{Nirenberg1953} on nonlinear elliptic equations in two dimensions.
Evans \cite{Evans1982} and Krylov \cite{Krylov1982} established interior H\"older estimates for second
derivatives of solutions to uniformly elliptic convex or concave equations. For the real Monge--Amp\`ere
equation, Caffarelli \cite{Caffarelli1990,Caffarelli1991} obtained interior $W^{2,p}$ and
$C^{2,\alpha}$ estimates for strictly convex solutions, while Trudinger--Wang
\cite{TrudingerWang2008} proved boundary regularity for the Monge--Amp\`ere and affine maximal
surface equations.

For the complex Monge--Amp\`ere equation, Bedford--Taylor \cite{BT1,BT2} developed pluripotential
theory, while Caffarelli--Kohn--Nirenberg--Spruck \cite{CKNS} solved the smooth Dirichlet problem
on smoothly bounded strictly pseudoconvex domains. On compact K\"ahler manifolds, Ko\l odziej
\cite{Kol} established uniform $L^\infty$ estimates for normalized solutions with $L^p$ density,
$p>1$. B{\l}ocki \cite{Blocki1} used local stability to prove a uniform bound, Ji \cite{Ji}
used De Giorgi iteration for $p>n$, and Guo--Phong--Tong
\cite{GuoPhongTong2023} developed a PDE approach to uniform estimates. Ko\l odziej \cite{Kol2008} also proved H\"older continuity of the potential for
$L^p$ densities with $p>1$, and Demailly--Dinew--Guedj--Pham--Ko\l odziej--Zeriahi
\cite{DDGHKZ} obtained further quantitative estimates. In degenerate settings, Eyssidieux--Guedj--Zeriahi \cite{EGZ2009} obtained bounded weak solutions
with applications to singular K\"ahler--Einstein metrics, while
Boucksom--Eyssidieux--Guedj--Zeriahi \cite{BEGZ2010} developed the full-mass theory in big cohomology
classes. On compact Hermitian manifolds, Tosatti--Weinkove \cite{TosattiWeinkove2010} solved the
smooth equation with a constant multiplying the prescribed density, and B{\l}ocki \cite{Blocki2}
proved the uniform estimate by a different method.

For higher regularity on compact K\"ahler manifolds, B{\l}ocki \cite{BlockiGradient} and Phong--Sturm
\cite{PhongSturmGradient} obtained gradient estimates without a prior Laplacian bound. Chen--He
\cite{CH} proved gradient and Laplacian estimates for $F\in W^{1,p}$, $p>2n$, while
Guo--Phong--Tong \cite{GuoPhongTongGradient} controlled the gradient using a weighted $L^{2n}$
bound for $\Na F$. Chen--Cheng \cite{CC2,CC} obtained gradient estimates under a Dini condition on
$F$. In particular, these estimates apply to $F\in C^\alpha$ without a bound for $\Na F$.
In the local theory, B{\l}ocki \cite{Blocki2000} studied interior regularity in convex domains,
while B{\l}ocki--Dinew \cite{BD} and He \cite{He} derived Laplacian bounds from integral control of
the Laplacian and additional regularity of the density. Chen--Wang \cite{CW} developed an a priori
approach under a bound for the complex Hessian. Under a Laplacian bound,
Tosatti--Wang--Weinkove--Yang \cite{TWWY} developed a unified approach in complex and almost complex
geometry, and Chu \cite{Chu} obtained the optimal H\"older exponent. For the complex Monge--Amp\`ere
equation with a positive $C^\alpha$ density, Dinew--Zhang--Zhang \cite{DZZ} proved interior
$C^{2,\alpha}$ regularity for $C^{1,1}$ solutions, and Wang \cite{WY} obtained the same regularity
under a bound for the complex Laplacian. For the same class of densities,
Li--Li--Zhang \cite{LLZ} proved interior $C^{2,\alpha}$ regularity assuming only that the
solution belongs to $C^{1,\beta}$ for $\beta>\beta_*(n,\alpha)$, with the threshold defined
in \eqref{eq:LLZ}.\\

However, a local $C^{\alpha}$ bound for the density alone cannot imply an interior $C^{2,\alpha}$ estimate.
Indeed, the local examples of B{\l}ocki \cite{Blo} and He \cite{He} show that such an
estimate can fail even when the density is smooth and positive and the potential is
bounded. On a compact K\"ahler manifold,  Chen--Cheng
\cite[Conjecture~1.1]{CC} asked whether a $C^\alpha$ bound for $F$ suffices for a
global $C^{2,\alpha}$ estimate. See also Chen--Cheng \cite{CC2}. The following theorem
answers this question.

\begin{theo}
\label{thm:main}Let $(M^n,g)$ be a connected compact K\"ahler manifold of complex dimension $n\geq2$.
For fixed $0<\alpha<1$ and $K>0$, we assume that $F$ is real-valued and satisfies
\beq
 \|F\|_{C^\alpha(M,g)}\leq K,\qquad
 \int_Me^F\,d\mu_g=1.
 \label{eq:data}
\eeq
If a bounded potential $\varphi\in\PSH(M,\omega_g)\cap L^\infty(M)$ satisfies
\beq
 (\omega_g+\pbp\varphi)^n=e^F\omega_g^n,\qquad
 \int_M\varphi\,d\mu_g=0,
 \label{eq:MA}
\eeq
in the pluripotential sense, then we have
\beq
 \varphi\in C^{2,\alpha}(M),\qquad
 \|\varphi\|_{C^{2,\alpha}(M,g)}
 \leq C(M,g,n,\alpha,K).
 \label{eq:main-estimate}
\eeq
Moreover, $\omega_\varphi>0$. The same conclusion holds with the normalization $\sup_M\varphi=0$, after a change of constant.
\end{theo}

To prove Theorem \ref{thm:main}, we seek a global
$C^{1,\beta}$ estimate with $\beta>\beta_*(n,\alpha)$ and apply the local theorem of Li--Li--Zhang \cite{LLZ} on the compact manifold.  Since bounds for the determinant alone
do not control the trace, the proof combines elliptic approximation and an auxiliary parabolic
equation with global Green and Sobolev estimates. In a series of papers,  Guo--Phong--Sturm and Guo--Phong--Song--Sturm
\cite{ GPSSs, GPSSd, GPSSs2, GPS} established uniform estimates for Green's functions, Sobolev
constants and diameters, with applications to complex Monge--Amp\`ere equations and extensions
to singular K\"ahler spaces. Guedj--T\^o \cite{GT} also obtained related estimates for Green's
functions, Sobolev constants and diameters in K\"ahler families. In our setting the volume densities
are uniformly bounded above and below, so these Green and Sobolev estimates  apply. We also use
the auxiliary flow, which is related to the parabolic regularization of Sz\'ekelyhidi--Tosatti \cite{ST}. Its time-dependent
forcing is chosen to control the error between the evolving potential and the elliptic solution.\\

We outline the idea of the proof. The first step is to smooth the density with the
background heat semigroup and solve the corresponding elliptic equations. The error is
controlled using the gradient estimate of Chen--Cheng \cite{CC} and cancellation in Green's
formula. Interpolating between this error estimate and the trace bounds for the smooth
approximations, we obtain an initial $C^{1,\beta}$ bound for $0<\beta<\alpha$. For the auxiliary flow, we start from the elliptic
solution and use the heat-regularized forcing. At time zero the potential has zero time
derivative, so its deviation from the initial data can be controlled. We next evolve the
elliptic approximations with the same forcing. By comparison and propagation of the trace
bounds, we can interpolate to obtain a uniform $C^{1,\beta}$ bound on the whole time interval.

With this uniform bound, we derive a weighted energy inequality using the evolving volume
measure. At positive times, the trace can then be bounded by parabolic Moser iteration,
independently of its initial size. Using this bound together with the error estimate for the
flow, we improve the H\"older exponent by interpolation. For uniformity up to time zero,
it suffices to restart the auxiliary equation with the regularized potentials as initial
data and with the original forcing. After finitely many repetitions, we reach the range of
exponents in the local theorem of Li--Li--Zhang \cite{LLZ} and apply it to obtain the full
real $C^{2,\alpha}$ estimate.

Finally, for general $F\in C^\alpha$, we approximate the density by smooth positive densities
with the same mass and uniformly bounded H\"older norms. By compactness, a subsequence
converges to the prescribed solution, as can be seen by elliptic approximation. We retain
the original H\"older exponent in the Hessian estimate when passing to the limit and obtain the strict
positivity by the lower bound for the determinant and the upper bound for the trace.\\

The paper is organized as follows. In Section~\ref{sec:preliminaries}, we fix the notation
and recall the Green function, Sobolev and gradient estimates, as well as the local regularity
theorem used in the proof. In Section~\ref{sec:elliptic-trace}, we first bound the potential
in $L^\infty$ and use heat estimates and interpolation to obtain a $C^{1,\beta}$ estimate
by elliptic approximation. We then construct the auxiliary flow and use weighted energy
estimates to control the trace at positive times. In Section~\ref{sec:conclusion}, we prove
Theorem~\ref{thm:main} by iterating the improvement of the H\"older exponent for smooth data
and passing to the limit for general $C^\alpha$ data. In Section~\ref{sec:local-obstructions},
we construct a sequence of smooth solutions with constant density whose interior
$C^{2,\alpha}$ norms diverge. We also prove interior estimates under a $C^{1,\beta}$ bound
for the potential or an $L^p$ bound for its trace, with suitable exponents.\\

\textbf{AI Declaration.}
The authors proposed the initial idea for this paper. Through discussions with ChatGPT (GPT-6 Astra), they developed the strategy for proving Proposition 3.3, using the symmetry of the heat kernel to exploit cancellation in the Green function representation and compare elliptic solutions at successive scales. They also used ChatGPT to search for relevant literature and refine the English exposition. The authors carefully checked the entire manuscript and take full responsibility for its contents.

\section{Preliminaries}
\label{sec:preliminaries}

In this section, we collect notation and results used in the later arguments.
Let $(M^n,g)$ be a connected compact K\"ahler manifold of complex dimension $n\geq2$ with the corresponding K\"ahler form $\omega_g$.
The normalized volume measure is denoted by $d\mu_g=V_g^{-1}\omega_g^n$, where $V_g=\int_M\omega_g^n$. The compatibility condition for \eqref{eq:intro-MA} is therefore $\int_Me^F\,d\mu_g=1$. For a smooth K\"ahler potential $\varphi$ with $\omega_\varphi>0$, we use the notation
$$
 \begin{gathered}
 g_\varphi=g+(\varphi_{i\bar j}),\qquad
 f_\varphi=\frac{\omega_\varphi^n}{\omega_g^n},\qquad
 d\mu_\varphi=V_g^{-1}\omega_\varphi^n=f_\varphi d\mu_g,\\
 \Delta_g a=g^{i\bar j}a_{i\bar j},\qquad
 \Delta_\varphi a=g_\varphi^{i\bar j}a_{i\bar j},\\
 v_\varphi=\tr_g g_\varphi=n+\Delta_g\varphi,\qquad
 q_\varphi=\tr_\varphi g.
 \end{gathered}
$$
Here $\tr_\varphi g$ denotes the trace with respect to $g_\varphi$. Note that $\int_Md\mu_\varphi=1$. For real functions, we use the real part of the Hermitian gradient pairing, so that
\beq
 \begin{gathered}
 \int_M a\Delta_\varphi b\,d\mu_\varphi
 =-\int_M\langle\Na a,\Na b\rangle_\varphi\,d\mu_\varphi,\\
 \Delta_\varphi(ab)
 =a\Delta_\varphi b+b\Delta_\varphi a
       +2\langle\Na a,\Na b\rangle_\varphi.
 \end{gathered}
 \label{eq:conventions}
\eeq
The heat operator is denoted by $P_t=e^{t\Delta_g}$, with its kernel taken relative to $d\mu_g$. Unless otherwise specified, all norms are taken with respect to $g$, and $C^\alpha=C^{0,\alpha}$ is equipped with the full norm. We write $C$ for a positive constant whose value may change. Unless otherwise stated, the constants in Sections~\ref{sec:preliminaries}--\ref{sec:conclusion} depend only on $M,g,n,\alpha,K$. For the local results, the dependence of the constants is given in each statement.

We recall the Green function and Sobolev estimates by Guo--Phong--Sturm \cite{GPS}, Guo--Phong--Song--Sturm \cite{GPSSs,GPSSd,GPSSs2} and Guedj--T\^o \cite{GT}. For a smooth K\"ahler potential satisfying
\beq
 0<c\leq f_\varphi\leq C_0,
 \label{eq:density-bounds}
\eeq
we normalize the Green function by
\beq
 \begin{gathered}
 (\Delta_\varphi G_x^\varphi)d\mu_\varphi
       =\delta_x-d\mu_\varphi,\qquad  \int_MG_x^\varphi\,d\mu_\varphi=0.
 \end{gathered}
 \label{eq:Green-normalization}
\eeq
Then we have the estimates
\beq
 \sup_MG_x^\varphi\leq C_G,\qquad
 \int_M|G_x^\varphi|^\nu\,d\mu_\varphi\leq C_\nu,
 \qquad 1<\nu<\frac n{n-1}.
 \label{eq:Green-estimate}
\eeq
The constant $C_G\geq1$ can be chosen uniformly in $x$ and $\varphi$. For the same range of $\nu$, the Sobolev inequality takes the form
\beq
 \left(\int_M|b|^{2\nu}\,d\mu_\varphi\right)^{1/\nu}
 \leq C_\nu\left(
 \int_M|\Na b|_\varphi^2\,d\mu_\varphi
       +\int_M|b|^2\,d\mu_\varphi\right).
 \label{eq:Sobolev-estimate}
\eeq
By \eqref{eq:density-bounds}, the measure in \eqref{eq:Green-estimate} can also be replaced by $d\mu_g$. These constants depend on $M,g,n,c,C_0$ and $\nu$, and do not require a bound for the trace of $g_\varphi$.

Chen--Cheng proved in \cite[Corollary 1.1]{CC} (see also Chen--Cheng \cite[Corollary 1.2]{CC2}) that
\beq
 \|\Na\varphi\|_\infty\leq C(M,g,n,\alpha,K)
 \label{eq:gradient-estimate}
\eeq
for smooth solutions of \eqref{eq:MA} satisfying \eqref{eq:data}. Their square-Dini condition holds because
$$
 \int_0^1\frac{(Kr^\alpha)^2}{r}\,dr
 =\frac{K^2}{2\alpha}<\infty.
$$
Thus their gradient estimate applies without a Hessian bound for $\varphi$.

The local regularity theorem used below is due to Li--Li--Zhang \cite[Theorem 1.2]{LLZ}. We assume that $\Phi\in\PSH(\Omega)\cap C^{1,\beta}(\Omega)$ is a weak solution of $\det(\Phi_{i\bar j})=a$ on a bounded domain $\Omega\subset\mathbb C^n$, where $a\geq\lambda>0$, $a\in C^\alpha$, and
\beq
\beta_*(n, \alpha):= 1-\frac{\alpha}{n(2+\alpha)-1}<\beta<1.
 \label{eq:LLZ}
\eeq
Then the local theorem gives a $C^{2,\alpha}$ bound for $\Phi$ on every $\Omega'\Subset\Omega$, depending on $n,\alpha,\beta,\lambda,\|a\|_{C^\alpha},\|\Phi\|_{C^{1,\beta}}$ and $\operatorname{dist}(\Omega',\partial\Omega)$.

\section{Estimates}
\label{sec:elliptic-trace}

\subsection{The $L^\infty$ estimate}
In this subsection, we establish the uniform bound for the potential needed in the gradient and
trace estimates below. The result is well known (cf. Ko\l odziej \cite[Section~2.3]{Kol}). Our proof uses
the Green function estimate.

\begin{lem}
\label{lem:C0}
If $\varphi$ is a smooth solution of \eqref{eq:MA}--\eqref{eq:data}, then $\|\varphi\|_\infty\leq C(M,g,n,K)$.
\end{lem}

\begin{proof}
We first prove a one-sided estimate. Let $a\geq0$ and let $h$ be smooth with
$\int_Mh\,d\mu_g=0$ and $\Delta_\varphi h\geq-a$. Assume that \eqref{eq:density-bounds} holds.
Since $\int_Md\mu_\varphi=\int_Md\mu_g=1$, it follows that $0<c\leq1$. By Green's formula and
the normalization \eqref{eq:Green-normalization}, we obtain
$$
 h(x)-\int_Mh\,d\mu_\varphi
 =\int_MG_x^\varphi\Delta_\varphi h\,d\mu_\varphi.
$$
The integral is unchanged when we subtract $C_G$ from the Green function, since
$\int_M\Delta_\varphi h\,d\mu_\varphi=0$. Using $G_x^\varphi-C_G\leq0$ and $\Delta_\varphi h\geq-a$, we obtain
$$
 \begin{aligned}
 h(x)-\int_Mh\,d\mu_\varphi
 &=\int_M(G_x^\varphi-C_G)\Delta_\varphi h\,d\mu_\varphi\\
 &\leq-a\int_M(G_x^\varphi-C_G)\,d\mu_\varphi
 =aC_G.
 \end{aligned}
$$
The last equality follows from the normalization $\int_MG_x^\varphi\,d\mu_\varphi=0$.
Set $m=\sup_Mh$. By the mean-zero condition, $\int_M(m-h)\,d\mu_g=m$. Using $m-h\geq0$ and
$f_\varphi\geq c$, we estimate the integral mean:
$$
 \begin{aligned}
 \int_Mh\,d\mu_\varphi
 &=m-\int_M(m-h)f_\varphi\,d\mu_g\\
 &\leq m-c\int_M(m-h)\,d\mu_g
 =(1-c)m.
 \end{aligned}
$$
Taking the supremum in the preceding Green estimate, we obtain $m\leq(1-c)m+aC_G$ and hence
$\sup_Mh\leq aC_G/c$.

We now apply this estimate to the solution $\varphi$. By \eqref{eq:data}, its density satisfies
$e^{-K}\leq e^F\leq e^K$. If $\lambda_1,\ldots,\lambda_n$ are the eigenvalues of $g_\varphi$ with
respect to $g$, then $\prod_i\lambda_i=e^F$. Combining the arithmetic--geometric mean inequality
with $e^s\geq1+s$, we obtain
$$
 \Delta_g\varphi
 =\sum_i\lambda_i-n
 \geq n(e^{F/n}-1)\geq F\geq-K.
$$
The one-sided estimate with the background metric and $h=\varphi$, $a=K$ and $c=1$ yields a
uniform upper bound for $\varphi$.

For the lower bound, we use $\prod_i\lambda_i^{-1}=e^{-F}$ and the same inequalities to obtain
$$
 \Delta_\varphi(-\varphi)
 =q_\varphi-n=\sum_i\lambda_i^{-1}-n
 \geq n(e^{-F/n}-1)\geq-F\geq-K.
$$
The mean-zero condition with respect to $d\mu_g$ holds for both choices of $h$. Since the density
lies between $e^{-K}$ and $e^K$, the constant $C_G$ is uniform by \eqref{eq:Green-estimate}.
Applying the one-sided estimate with the metric $\omega_\varphi$ and $h=-\varphi$, taking $a=K$
and $c=e^{-K}$, we obtain a uniform upper bound for $-\varphi$. All constants depend only on
$M,g,n,K$.
\end{proof}

\subsection{Heat estimates and interpolation}
\label{Heat estimates and interpolation}
In this subsection, we recall heat kernel estimates and prove an interpolation lemma. Let $k_t$
denote the heat kernel of the fixed background metric relative to $d\mu_g$. As in
Grigor'yan \cite[Chapters~7 and 8]{Grig}, the kernel is symmetric and nonnegative and preserves
mass. In particular,
\begin{align}\label{mass identities}
 k_t(x,y)=k_t(y,x)\geq0,\quad
\int_M k_t(y,z)\,d\mu_g(y)
=
\int_M k_t(z,y)\,d\mu_g(y)
=1.
\end{align}
The heat semigroup therefore satisfies
\begin{align}\label{pt a}
P_ta(x)=\int_M k_t(x,y)a(y)\,d\mu_g(y),\quad
P_t1=1,\quad \int_M P_ta\,d\mu_g=\int_Ma\,d\mu_g.
\end{align}
Moreover, $P_{s+t}=P_sP_t$ and $\|P_ta\|_\infty\leq\|a\|_\infty$ for $s,t>0$. In particular, if $c\leq a\leq C_0$, then $c\leq P_ta\leq C_0$.

We next recall the H\"older smoothing estimates for the heat semigroup. By localizing in a finite
atlas, one can apply the analytic semigroup estimates in H\"older spaces from
Lunardi \cite[Sections~2.2 and 3.1]{Lunardi} to $\Delta_g$ on the fixed smooth compact manifold.
For $0<\alpha<1$, $a\in C^\alpha(M)$, $0<r\leq1$, and $t=r^2$, we obtain
\begin{equation}
\begin{split}
 \|P_{r^2}a-a\|_\infty
 &\leq C r^\alpha\|a\|_{C^\alpha},\quad
 \|P_{r^2}a\|_{C^\alpha}
 \leq C\|a\|_{C^\alpha},\\
 \|\Na P_{r^2}a\|_\infty
 &\leq C r^{\alpha-1}\|a\|_{C^\alpha},\quad
 \|\Delta_gP_{r^2}a\|_\infty
 \leq C r^{\alpha-2}\|a\|_{C^\alpha}.
 \label{eq:heat-alpha}
 \end{split}
\end{equation}
For $0<\beta<1$, the corresponding estimates for $C^{1,\beta}$ functions are
\beq
 \begin{aligned}
 \|\Na(a-P_{r^2}a)\|_\infty
 \leq C r^\beta\|a\|_{C^{1,\beta}},\quad
 \|\Delta_gP_{r^2}a\|_\infty
 \leq C r^{\beta-1}\|a\|_{C^{1,\beta}}.
 \end{aligned}
 \label{eq:heat-gradient}
\eeq

For the Gaussian bounds, we use the two-sided heat kernel estimates of
Davies \cite[Chapter~5]{Davies} and obtain the derivative bounds from the local parabolic
estimates in Grigor'yan \cite[Sections~6.4 and 7.1]{Grig}. On the fixed background, the estimates
take the following form for $0<t\leq1$.
 \begin{align}\label{Gaussian bounds}
 \begin{split}
C^{-1}t^{-n}\exp\!\left(-\frac{C\,d_g(y,z)^2}{t}\right)
 \leq k_t(y,z)
 &\leq C t^{-n}\exp\!\left(-\frac{d_g(y,z)^2}{C t}\right),\\
 |\Na_x k_t(x,y)|
 &\leq Ct^{-n-1/2}\exp\left(-\frac{d_g(x,y)^2}{Ct}\right),\\
 |\partial_t k_t(x,y)|=|\Delta_{g,x}k_t(x,y)|
 &\leq Ct^{-n-1}\exp\left(-\frac{d_g(x,y)^2}{Ct}\right).
 \end{split}
 \end{align}
Here $t^{-n}$ corresponds to the real dimension $2n$, and the normalization of the volume measure
affects only the background constants.
Combining the Gaussian upper bounds with the background volume bound
$\mu_g(B_g(x,\rho))\leq C\rho^{2n}$, we obtain the following moment estimate for every fixed
$\sigma\geq0$:
 \begin{align}\label{kernel mass}
 &\int_M\bigl(k_t(x,y)+\sqrt t\,|\Na_xk_t(x,y)|
                  +t|\partial_tk_t(x,y)|\bigr)
                  d_g(x,y)^\sigma\,d\mu_g(y)\leq C_\sigma t^{\sigma/2}.
 \end{align}
Since $\partial_tP_ta=\Delta_gP_ta$, the last estimate in \eqref{eq:heat-alpha} can equivalently be written as
\beq
 \|\partial_tP_ta\|_\infty
 =\|\Delta_gP_ta\|_\infty
 \leq Ct^{\alpha/2-1}\|a\|_{C^\alpha}.
 \label{eq:heat-time}
\eeq
In \eqref{eq:heat-time}, $a\in C^\alpha(M)$ and $0<t\leq1$. All constants in these heat estimates
depend only on $M,g,n$ and the indicated H\"older exponents. The constant in \eqref{kernel mass}
also depends on $\sigma$.

We also use elliptic regularity for $\Delta_g$ and Sobolev embedding in real dimension $2n$.
For smooth $a$, the estimates of Gilbarg--Trudinger \cite[Chapters~7 and 9]{GilTru} yield
\beq
 \begin{gathered}
 \|a\|_{W^{2,p}}
 \leq C_p(\|\Delta_ga\|_{L^p(d\mu_g)}+\|a\|_\infty),\\
 \|a\|_{C^{1,\gamma}}\leq C_{p,\gamma}\|a\|_{W^{2,p}},
 \qquad p>\frac{2n}{1-\gamma},\quad 0<\gamma<1.
 \end{gathered}
 \label{eq:background-elliptic}
\eeq
The first estimate holds for every $1<p<\infty$ by applying the interior $W^{2,p}$ estimate on
a finite collection of coordinate balls. For the second estimate, we use
$W^{2,p}\hookrightarrow C^{1,1-2n/p}$ when $p>2n$, since the stated condition is equivalent to
$\gamma<1-2n/p$. The constants depend only on the fixed background and the displayed exponents.

The following interpolation lemma relates estimates for smooth approximations to the regularity
of their limit. In the proof, we estimate the differences between successive dyadic
approximations and sum the resulting geometric series.

\begin{lem}
\label{lem:interpolation}
Fix $a,b>0$, $0<\gamma<1$, and $0<\varepsilon_0\leq1$. Suppose that a family of smooth functions
$\varphi_\varepsilon$ and a continuous function $\varphi$ satisfy
$$
 \|\varphi_\varepsilon-\varphi\|_\infty\leq C\varepsilon^a,
 \qquad
 \|\varphi_\varepsilon\|_{C^{1,\gamma}}\leq C\varepsilon^{-b}
 \quad (0<\varepsilon\leq\varepsilon_0).
$$
Then $\varphi$ has a uniform $C^{1,\beta}$ bound whenever
\beq
 0<\beta<\frac{a\gamma-b}{a+b}.
 \label{eq:interpolation-range}
\eeq
If the assumptions hold uniformly in an additional parameter, then the conclusion is uniform in
that parameter. If the assumptions hold for every $0<\gamma<1$, every $0<\beta<(a-b)/(a+b)$
is admissible. The bound may depend on the constants in the assumptions,
$\varepsilon_0,a,b,\gamma$, and the chosen exponent $\beta$.
\end{lem}

\begin{proof}
Fix $\beta$ satisfying \eqref{eq:interpolation-range}. Since $b>0$, it follows that
$0<\beta<\gamma$. To specify the exponents and the dependence of the constant, we first prove
the H\"older interpolation inequality in the form used here. See Gilbarg--Trudinger
\cite[Section~6.8]{GilTru} for the standard interpolation inequality.
\beq
 \|h\|_{C^{1,\beta}}
 \leq C\|h\|_\infty^{(\gamma-\beta)/(1+\gamma)}
         \|h\|_{C^{1,\gamma}}^{(1+\beta)/(1+\gamma)}.
 \label{eq:Holder-interpolation}
\eeq
Here $h$ is any smooth function on $M$, and $C$ depends only on $M,g,\beta,\gamma$. Set
$A=\|h\|_\infty$ and $H=\|h\|_{C^{1,\gamma}}$. We may assume $A>0$, since otherwise $h=0$.
Choose a finite collection of coordinate balls whose smaller concentric balls cover $M$. Fix
$0<r_0\leq1$ so that, for every point in a smaller ball, the coordinate segments of length
$0<r\leq r_0$ used in the difference quotients lie in the corresponding larger ball. Using Taylor's
formula, we have
$$
 \left|D_i h(x)-\frac{h(x+re_i)-h(x)}r\right|
 \leq C H r^\gamma,
$$
where $D_i$ denotes coordinate differentiation and $e_i$ is the corresponding coordinate vector. Consequently,
$$
 \|\Na h\|_\infty\leq C(r^{-1}A+r^\gamma H).
$$
Since $A\leq H$, we may take $r=r_0(A/H)^{1/(1+\gamma)}$ and obtain
$$
 \|\Na h\|_\infty
 \leq C A^{\gamma/(1+\gamma)}H^{1/(1+\gamma)}.
$$

To estimate the H\"older seminorm, we take $x,y$ in the same smaller coordinate ball. Then
$$
 |Dh(x)-Dh(y)|
 \leq C\min\{\|\Na h\|_\infty,H|x-y|^\gamma\}.
$$
Dividing by $|x-y|^\beta$ and using $\min\{s,t\}\leq s^{1-\beta/\gamma}t^{\beta/\gamma}$ for $s,t\geq0$, we obtain, on this finite atlas,
$$
 [\Na h]_{C^\beta}
 \leq C\|\Na h\|_\infty^{1-\beta/\gamma}H^{\beta/\gamma}
 \leq C A^{(\gamma-\beta)/(1+\gamma)}
             H^{(1+\beta)/(1+\gamma)}.
$$
The coordinate norms are equivalent to the background norms, with constants depending only on
the fixed atlas and $g$. Since $A\leq H$, the same right-hand side also bounds
$\|h\|_\infty+\|\Na h\|_\infty$. This proves \eqref{eq:Holder-interpolation} for the full norms.

For the dyadic approximation, set $\varepsilon_j=\varepsilon_0 2^{-j}$ and
$d_j=\varphi_{\varepsilon_{j+1}}-\varphi_{\varepsilon_j}$ for $j\geq0$. The triangle inequality
and the assumptions imply
$$
 \begin{aligned}
 \|d_j\|_\infty
 &\leq\|\varphi_{\varepsilon_{j+1}}-\varphi\|_\infty
       +\|\varphi_{\varepsilon_j}-\varphi\|_\infty
 \leq C\varepsilon_j^a,\\
 \|d_j\|_{C^{1,\gamma}}
 &\leq\|\varphi_{\varepsilon_{j+1}}\|_{C^{1,\gamma}}
       +\|\varphi_{\varepsilon_j}\|_{C^{1,\gamma}}
 \leq C\varepsilon_j^{-b}.
 \end{aligned}
$$
Applying \eqref{eq:Holder-interpolation} to $d_j$, we obtain
$$
 \|d_j\|_{C^{1,\beta}}
 \leq C\varepsilon_j^\eta,\qquad
 \eta=\frac{a(\gamma-\beta)-b(1+\beta)}{1+\gamma}
     =\frac{a\gamma-b-(a+b)\beta}{1+\gamma}>0.
$$
By \eqref{eq:interpolation-range}, the exponent $\eta$ is positive. Hence
$$
 \sum_{j=0}^\infty\|d_j\|_{C^{1,\beta}}
 \leq C\varepsilon_0^\eta\sum_{j=0}^\infty2^{-j\eta}
 =\frac{C\varepsilon_0^\eta}{1-2^{-\eta}}.
$$
Since $C^{1,\beta}(M)$ is a Banach space, the series
$\varphi_{\varepsilon_0}+\sum_{j=0}^\infty d_j$ converges in $C^{1,\beta}$. Its partial sums are
$\varphi_{\varepsilon_{j+1}}$, which converge uniformly to $\varphi$. Thus the sum equals
$\varphi$, and we obtain
$$
 \|\varphi\|_{C^{1,\beta}}
 \leq C\varepsilon_0^{-b}
       +\frac{C\varepsilon_0^\eta}{1-2^{-\eta}}.
$$
All constants are uniform in an additional parameter whenever the constants in the assumptions are uniform.

Finally, if the assumptions hold for every $0<\gamma<1$, then for any fixed $0<\beta<(a-b)/(a+b)$ we choose
$$
 \frac{b+(a+b)\beta}{a}<\gamma<1.
$$
The restriction on $\beta$ ensures that this interval is nonempty. For any such $\gamma$,
\eqref{eq:interpolation-range} holds, so the preceding argument proves the last assertion.
\end{proof}

\subsection{Approximation by elliptic solutions}

In this subsection, we establish the initial $C^{1,\beta}$ estimate for the potential. We  first
 smooth the density by the background heat operator and estimate the trace of the
corresponding elliptic solutions. By using cancellation in Green's formula, we improve the
approximation estimate enough to apply Lemma~\ref{lem:interpolation}.

\begin{prop}
\label{prop:density-approximation}
For $F$ satisfying \eqref{eq:data} and a bounded mean-zero pluripotential solution $\varphi$ of
\eqref{eq:MA}, we set
$$
 f_\varepsilon=P_{\varepsilon^2}e^F,\qquad
 F_\varepsilon=\log f_\varepsilon,\qquad 0<\varepsilon\leq1/4.
$$
The smooth normalized solution with density $f_\varepsilon$ is denoted by $\varphi_\varepsilon$. For
every $\delta>0$, it satisfies
\beq
 \|\varphi_\varepsilon-\varphi\|_\infty
 \leq C\varepsilon^{1+\alpha}(1+|\log\varepsilon|),\qquad
 \|n+\Delta_g\varphi_\varepsilon\|_\infty
 \leq C_\delta\varepsilon^{-(1-\alpha)(1+\delta)}.
 \label{eq:density-approximation}
\eeq
In particular, $\|\varphi\|_{C^{1,\beta}}\leq C_\beta$ for each $0<\beta<\alpha$.
\end{prop}

The proof of Proposition~\ref{prop:density-approximation} requires a trace estimate for the smooth
elliptic solutions and a bound for their approximation error. For the trace estimate, we first
derive differential inequalities that will also be used in the auxiliary flow.

\begin{lem}
\label{lem:differential}
For a smooth K\"ahler potential $\varphi$, we set
$$
 H=\log f_\varphi,\qquad v=n+\Delta_g\varphi,\qquad
 q=\tr_\varphi g,\qquad S=|\Na\varphi|_g^2.
$$
At the point under consideration, choose $g$-normal holomorphic coordinates in which $g_\varphi$
is diagonal, with eigenvalues $\lambda_i>0$. In these coordinates, we define
$$
 \mathcal T
 =\sum_{i,j}\frac{|\varphi_{ij}|^2+|\varphi_{i\bar j}|^2}
                       {\lambda_i}.
$$
Here all covariant derivatives are taken with respect to $g$. For a constant $C_g$ depending only on
the background metric, the following inequalities hold:
\beq
 \Delta_\varphi\log v\geq\frac{\Delta_gH}{v}-C_gq,
 \label{eq:trace-differential}
\eeq
\beq
 \Delta_\varphi S
 \geq\mathcal T-C_gSq
       +2\langle\Na\varphi,\Na H\rangle_g,
 \label{eq:gradient-differential}
\eeq
and, for every real smooth function $a$,
\beq
 \mathcal T\geq v+q-2n,\qquad
 |\Na a|_g^2\leq v|\Na a|_\varphi^2.
 \label{eq:algebraic-comparisons}
\eeq
\end{lem}

\begin{proof}
In computing $\Delta_\varphi v$, we follow Yau \cite{Yau}. Differentiating
$H=\log\det g_\varphi-\log\det g$ twice in the chosen coordinates and commuting covariant
derivatives, we obtain
$$
 \Delta_\varphi v
 =\Delta_gH
  +\sum_{i,j,k}\frac{|\varphi_{i\bar j k}|^2}
                         {\lambda_i\lambda_j}
  +\mathcal E,
 \qquad |\mathcal E|\leq C_gvq.
$$
To bound the curvature terms, we use the background curvature bound and the identity
$\sum_{i,k}\lambda_k/\lambda_i=vq$. The remaining constant terms are absorbed using $vq\geq n^2$.
By symmetry of the holomorphic indices and the Cauchy--Schwarz inequality, we obtain
$$
 \frac{|\Na v|_\varphi^2}{v}
 =\sum_k\frac{|\sum_i\varphi_{i\bar i k}|^2}{\lambda_k v}
 \leq\sum_{i,k}\frac{|\varphi_{i\bar i k}|^2}
                         {\lambda_k\lambda_i}
 \leq\sum_{i,j,k}\frac{|\varphi_{i\bar j k}|^2}
                         {\lambda_i\lambda_j}.
$$
The last inequality follows from the identity $\varphi_{i\bar i k}=\varphi_{k\bar i i}$. We then
differentiate $\log v$ to obtain \eqref{eq:trace-differential}.
For the gradient computation, we use the calculation of Chen--Cheng \cite[p.~2076]{CC2}.
By differentiating $S$ and using $\sum_i\lambda_i^{-1}\varphi_{i\bar i j}=H_j$, we have
$$
 \Delta_\varphi S
 =\mathcal T+2\langle\Na\varphi,\Na H\rangle_g+\mathcal E_1,
 \qquad |\mathcal E_1|\leq C_gSq.
$$
This proves \eqref{eq:gradient-differential}.
It remains to prove \eqref{eq:algebraic-comparisons}. For the first inequality, it suffices to
retain the diagonal mixed derivatives and use
$$
 \sum_i\frac{|\lambda_i-1|^2}{\lambda_i}
 =v+q-2n.
$$
The second inequality follows from $\lambda_i\leq v$ in each gradient component.
\end{proof}

We now apply these differential inequalities to a weighted trace. By integrating by parts in the
term containing $\Delta_gF$, one obtains an integral estimate in terms of the first derivatives
of $F$. Using Moser iteration we have the supremum bound needed for the smooth elliptic
approximations.

\begin{lem}
\label{lem:elliptic-trace}
Assume that a smooth solution of \eqref{eq:MA} satisfies $\|F\|_\infty\leq K$ and
$\|\varphi\|_{C^1}\leq B$. For $D=1+\|\Na F\|_\infty$, the following estimates hold for every finite
$p\geq1$ and every $\delta>0$:
\beq
 \|n+\Delta_g\varphi\|_{L^p(d\mu_g)}\leq C_pD,\qquad
 \|n+\Delta_g\varphi\|_\infty\leq C_\delta D^{1+\delta}.
 \label{eq:elliptic-trace-bounds}
\eeq
The constants depend only on $M,g,n,K,B$ and the exponents in these estimates.
\end{lem}

\begin{proof}
We divide the proof into three steps.

\paragraph{Step 1.}
With the notation of Lemma~\ref{lem:differential}, we choose $A\geq C_g(B^2+1)+1$ and set
$$
 Q=S-A\varphi,\qquad E=e^Q,\qquad W=Ev.
$$
The function $E$ satisfies $C^{-1}\leq E\leq C$, with constants depending only on $M,g,n,K,B$. Using $S\leq B^2$ and $\Delta_\varphi\varphi=n-q$ in Lemma~\ref{lem:differential}, we obtain
$$
 \begin{aligned}
 \Delta_\varphi\log W
 &=\Delta_\varphi(S-A\varphi+\log v)\\
 &\geq\mathcal T+\bigl(A-C_g(S+1)\bigr)q-An
       +2\langle\Na\varphi,\Na F\rangle_g
       +\frac{\Delta_gF}{v}\\
 &\geq\mathcal T+q-CD+\frac{\Delta_gF}{v}.
 \end{aligned}
$$
Using $\Delta_\varphi W=W\Delta_\varphi\log W+W|\Na\log W|_\varphi^2$ and retaining half of
$\mathcal T$, we apply $\mathcal T\geq v+q-2n$ to the other half. Since $Wv=E^{-1}W^2$, we obtain
\beq
 \Delta_\varphi W
 \geq\frac12W\mathcal T+cW^2-CDW+E\Delta_gF.
 \label{eq:elliptic-W}
\eeq
For the integration by parts in the next step, we also need the gradient estimates
\beq
 |\Na Q|_g^2\leq C(1+v\mathcal T),\qquad
 |\Na W|_g^2\leq v|\Na W|_\varphi^2.
 \label{eq:elliptic-gradients}
\eeq
Indeed, we have $|\Na S|_g\leq C B|\Na^2\varphi|_g$ and $|\Na^2\varphi|_g^2\leq Cv\mathcal T$.

\paragraph{Step 2.}
For \eqref{eq:elliptic-W}, we use the integral method of Chen--He \cite[Section~2]{CH}.
Multiply \eqref{eq:elliptic-W} by $W^{m-1}$ for real $m\geq2$ and integrate by parts with
respect to $d\mu_\varphi$ to obtain
\beqn
 &&(m-1)\int_MW^{m-2}|\Na W|_\varphi^2\,d\mu_\varphi
  +\frac12\int_MW^m\mathcal T\,d\mu_\varphi
  +c\int_MW^{m+1}\,d\mu_\varphi \no\\
 &&\hspace{15mm}\leq CD\int_MW^m\,d\mu_\varphi+I,
 \qquad
 I=-\int_MEW^{m-1}\Delta_gF\,d\mu_\varphi.
 \label{eq:elliptic-energy-start}
\eeqn
To estimate the last integral, write $d\mu_\varphi=e^Fd\mu_g$. Integrating by parts, we
differentiate both the density and $EW^{m-1}$ and obtain
\beqs
 I&=&(m-1)\int_MEW^{m-2}
                    \langle\Na W,\Na F\rangle_g\,d\mu_\varphi\\
 &&+\int_MEW^{m-1}
                    \langle\Na Q,\Na F\rangle_g\,d\mu_\varphi
   +\int_MEW^{m-1}|\Na F|_g^2\,d\mu_\varphi.
\eeqs

We estimate the three terms in $I$ using \eqref{eq:elliptic-gradients} and Young's inequality. The
first integral is bounded above by
$$
 \frac{m-1}{2}\int_MW^{m-2}|\Na W|_\varphi^2\,d\mu_\varphi
       +CmD^2\int_MW^{m-1}\,d\mu_\varphi.
$$
The second integral is bounded above by
$$
 \frac14\int_MW^m\mathcal T\,d\mu_\varphi
       +CD^2\int_MW^{m-1}\,d\mu_\varphi.
$$
Here the term without $\mathcal T$ is controlled using $D\leq D^2$. The third integral is at most
$CD^2\int_MW^{m-1}\,d\mu_\varphi$. Combining these bounds, we obtain
\beqn
 &&\frac{m-1}{2}\int_MW^{m-2}|\Na W|_\varphi^2\,d\mu_\varphi
  +\frac14\int_MW^m\mathcal T\,d\mu_\varphi
  +c\int_MW^{m+1}\,d\mu_\varphi \no\\
 &&\hspace{12mm}\leq
 CD\int_MW^m\,d\mu_\varphi
       +CmD^2\int_MW^{m-1}\,d\mu_\varphi.
 \label{eq:elliptic-energy}
\eeqn

For $X=\|W\|_{L^{m+1}(d\mu_\varphi)}$, it follows from the normalization $\int_Md\mu_\varphi=1$ that
$$
 cX^{m+1}\leq CDX^m+CmD^2X^{m-1}.
$$
After division by $X^{m-1}>0$, the resulting quadratic inequality implies $X\leq C\sqrt mD$.
The first estimate in \eqref{eq:elliptic-trace-bounds} follows from H\"older's inequality and the
equivalence of the two measures.

\paragraph{Step 3.}
To prove the supremum estimate, we set $m=p+1$ in \eqref{eq:elliptic-energy}. For $p\geq2$, it
follows from \eqref{eq:elliptic-gradients}, $v=E^{-1}W$ and the equivalence of the two measures that
\beq
 \int_M|\Na(W^{p/2})|_g^2\,d\mu_g
 \leq Cp^2\int_M(DW+D^2)W^p\,d\mu_g.
 \label{eq:elliptic-Moser-energy}
\eeq
We use the background Sobolev inequality in real dimension $2n$, as in Gilbarg--Trudinger
\cite[Chapter~7]{GilTru}. Since $D\geq1$, its zeroth-order term can be included by increasing $C$.
Set $\chi=n/(n-1)$ and $b=W^{p/2}$. By the integral estimate from Step~2,
$\|DW+D^2\|_{L^{2n}(d\mu_g)}\leq CD^2$. With this bound, H\"older's inequality yields
$$
 \|b\|_{L^{2\chi}}^2
 \leq Cp^2D^2\|b\|_{L^{4n/(2n-1)}}^2
 \leq Cp^2D^2\|b\|_{L^{2\chi}}\|b\|_{L^2}.
$$
The last inequality follows by interpolation between $L^2$ and $L^{2\chi}$ with equal weights. Dividing by $\|b\|_{L^{2\chi}}$ and squaring, we obtain
$$
 \|W\|_{L^{\chi p}}\leq (Cp^4D^4)^{1/p}\|W\|_{L^p}.
$$
For the Moser iteration, we take $p_j=p_0\chi^j$. Since $\sum_{j\geq0}p_j^{-1}=n/p_0$ and
$\sum_{j\geq0}p_j^{-1}\log p_j<\infty$, it follows that
$$
 \|W\|_\infty
 \leq C_{p_0}D^{4n/p_0}\|W\|_{L^{p_0}}
 \leq C_{p_0}D^{1+4n/p_0}.
$$
This completes the proof when $p_0>\max\{2,4n/\delta\}$, since $C^{-1}v\leq W\leq Cv$.
\end{proof}

Having established the trace estimate, we turn to the approximation error in
Proposition~\ref{prop:density-approximation}. The comparison of solutions at successive scales
uses Green's formula. To prepare for this comparison, we estimate the background gradient of a
truncated Green function and the integral of its tail using the bounds for the density and the
gradient of the potential.

\begin{lem}
\label{lem:truncated-Green}
For a smooth K\"ahler potential $\varphi$ satisfying \eqref{eq:density-bounds} and
$\|\Na\varphi\|_\infty\leq B$, we define, for $x\in M$ and $N\geq2$,
$$
 Z=C_G+1-G_x^\varphi,\qquad Z_N=\min\{Z,N\}.
$$
These functions satisfy
\beq
 \int_M|\Na Z_N|_g\,d\mu_g\leq C(1+\log N),\qquad
 \int_M(Z-Z_N)\,d\mu_g\leq C_\nu N^{1-\nu},
 \quad 1<\nu<\frac n{n-1}.
 \label{eq:Green-truncations}
\eeq
The constant $C$ depends on $M,g,n,c,C_0,B$, and $C_\nu$ may also depend on $\nu$. Neither constant
depends on $x,N$ or the trace of $g_\varphi$.
\end{lem}

\begin{proof}
It follows from the definition \eqref{eq:Green-normalization} and estimate \eqref{eq:Green-estimate}
that
\begin{align*}
Z\ge1,\qquad
(\Delta_\varphi Z)d\mu_\varphi=d\mu_\varphi-\delta_x,
\qquad
\int_M Z\,d\mu_\varphi=C_G+1.
\end{align*}
Consequently, we have
\[
\int_M Z^\theta\,d\mu_g
\le \int_M Z\,d\mu_g
\le c^{-1}(C_G+1),
\qquad\text{ uniformly for }0<\theta<1.
\]

For each fixed smooth metric $g_\varphi$, the function $Z$ is smooth on $M\setminus\{x\}$ and
satisfies $Z(y)\to+\infty$ as $y\to x$. Thus, for each finite $R>1$, the function
$ Z_R=\min\{Z,R\} $ is constant on a neighborhood of $x$ and is Lipschitz on $M$. The Lipschitz
chain rule implies
\[
 \nabla Z_R=\mathbf{1}_{\{Z<R\}}\nabla Z
 \qquad\text{almost everywhere}.
\]
The values of gradients on a level set are understood in this almost-everywhere sense.

\paragraph{Step 1.}
For $0<\theta<1$, the function $Z^{\theta-1}$ tends to zero at the pole. Testing the equation for
$Z$ with a truncation of this function and then passing to the limit, we will obtain
\beq
 \int_M Z^{\theta-2}|\Na Z|_\varphi^2\,d\mu_\varphi
 =\frac1{1-\theta}\int_MZ^{\theta-1}\,d\mu_\varphi
 \leq\frac1{1-\theta}.
 \label{eq:Green-weighted-energy}
\eeq

To justify the test at the pole, we use a truncation. For a test function $b$ that is smooth away
from the pole and vanishes near it, the weak form of the equation
$(\Delta_\varphi Z)d\mu_\varphi=d\mu_\varphi-\delta_x$ reads
\begin{equation}\label{eq:weak}
-\int_M\langle\nabla b,\nabla Z\rangle_\varphi\,d\mu_\varphi
=\int_M b\,d\mu_\varphi-b(x).
\end{equation}
We fix $R>1$ and define
\[
b_R=\bigl(Z^{\theta-1}-R^{\theta-1}\bigr)_+
=\begin{cases}
Z^{\theta-1}-R^{\theta-1},& Z<R,\\
0,&Z\ge R.
\end{cases}
\]
Since $Z(y)\to\infty$ as $y\to x$, the function $b_R$ vanishes on a neighborhood of $x$. It is
Lipschitz on $M$. By mollifying in coordinate charts away from $x$, we approximate $b_R$ in
$W^{1,1}$ by smooth functions supported in a fixed compact subset of $M\setminus\{x\}$. The
gradient of $Z$ is bounded there, so \eqref{eq:weak} passes to the limit with $b=b_R$ and
$b_R(x)=0$. By the chain rule for Lipschitz truncations, we have almost everywhere
\[
\nabla b_R=(\theta-1)\mathbf{1}_{\{Z<R\}}Z^{\theta-2}\nabla Z.
\]
Therefore, we have
\begin{equation}\label{eq:truncated}
(1-\theta)\int_{\{Z<R\}}Z^{\theta-2}|\nabla Z|_\varphi^2\,d\mu_\varphi
=\int_M\bigl(Z^{\theta-1}-R^{\theta-1}\bigr)_+\,d\mu_\varphi.
\end{equation}

Both integrands are nonnegative. As $R\to\infty$, the sets $\{Z<R\}$ increase to $M\setminus\{x\}$
and the right-hand integrand increases pointwise to $Z^{\theta-1}$ away from the pole. By monotone
convergence,
\[
(1-\theta)\int_M Z^{\theta-2}|\nabla Z|_\varphi^2\,d\mu_\varphi
=\int_M Z^{\theta-1}\,d\mu_\varphi.
\]
Finally, since $Z\ge1$ and $\theta-1<0$, we have $ 0\leq Z^{\theta-1}\le1$ and
$\int_M Z^{\theta-1}\,d\mu_\varphi\le\int_M1\,d\mu_\varphi=1. $ This proves
\eqref{eq:Green-weighted-energy}.

\paragraph{Step 2.}
We next estimate the weighted trace by combining \eqref{eq:Green-weighted-energy} with the
gradient bound for $\varphi$. Set $v=\operatorname{tr}_g g_\varphi=n+\Delta_g\varphi>0$. The
claim is that
\begin{align}\label{J theta bound}
 J_\theta=\int_MvZ^\theta\,d\mu_g\leq\frac{C}{1-\theta},
 \qquad 0<\theta<1,
\end{align}
where $C$ is independent of the pointwise size of $v$.

For fixed $R$, we first truncate by setting
\[
 J_{\theta,R}=\int_MvZ_R^\theta\,d\mu_g.
\]
Integrating by parts with respect to the background metric, we obtain
\[
 \begin{aligned}
 J_{\theta,R}
 &=n\int_MZ_R^\theta\,d\mu_g
      +\int_MZ_R^\theta\Delta_g\varphi\,d\mu_g
 =n\int_MZ_R^\theta\,d\mu_g
      -\int_M\langle\nabla(Z_R^\theta),\nabla\varphi\rangle_g
                                                  \,d\mu_g\\
 &=n\int_MZ_R^\theta\,d\mu_g
      -\theta\int_{\{Z<R\}}Z^{\theta-1}
           \langle\nabla\varphi,\nabla Z\rangle_g\,d\mu_g.
 \end{aligned}
\]
Since $1\leq Z_R\leq Z$ and $0<\theta<1$, we have $Z_R^\theta\leq Z$. By the measure comparison
and the Green function normalization, the first term is therefore at most $n c^{-1}(C_G+1)$.
For the second term, combine $|\nabla Z|_g\leq v^{1/2}|\nabla Z|_\varphi$ with the gradient bound
$|\nabla\varphi|_g\leq B$ to obtain
\[
 J_{\theta,R}
 \leq n c^{-1}(C_G+1)
  +B\theta\int_{\{Z<R\}}
        Z^{\theta-1}v^{1/2}|\nabla Z|_\varphi\,d\mu_g.
\]

The weights satisfy
\[
 Z^{\theta-1}v^{1/2}|\nabla Z|_\varphi
 =(vZ^\theta)^{1/2}
   \bigl(Z^{\theta-2}|\nabla Z|_\varphi^2\bigr)^{1/2}.
\]
Using the Cauchy--Schwarz inequality, the measure comparison and \eqref{eq:Green-weighted-energy},
we obtain
\[
 \begin{aligned}
 \int_{\{Z<R\}}Z^{\theta-1}v^{1/2}|\nabla Z|_\varphi\,d\mu_g
&\leq
  \left(\int_{\{Z<R\}}vZ^\theta\,d\mu_g\right)^{1/2}
  \left(\int_{\{Z<R\}}Z^{\theta-2}|\nabla Z|_\varphi^2
                                             \,d\mu_g\right)^{1/2}\\
 &\quad\leq c^{-1/2}J_{\theta,R}^{1/2}
  \left(\int_MZ^{\theta-2}|\nabla Z|_\varphi^2
                                    \,d\mu_\varphi\right)^{1/2}
\leq\frac{c^{-1/2}J_{\theta,R}^{1/2}}{\sqrt{1-\theta}}.
 \end{aligned}
\]
The first factor is bounded by $J_{\theta,R}^{1/2}$ because $Z_R=Z$ on $\{Z<R\}$ and $v>0$.
Only the second factor requires a change of measure.

Using $\theta\leq1$ and Young's inequality, we conclude that
\[
 J_{\theta,R}
 \leq n c^{-1}(C_G+1)
       +\frac{B c^{-1/2}}{\sqrt{1-\theta}}J_{\theta,R}^{1/2}
 \leq n c^{-1}(C_G+1)
+\frac12J_{\theta,R}+\frac{B^2}{2c(1-\theta)}.
\]
Absorb $J_{\theta,R}/2$ into the left-hand side to obtain
\[
 J_{\theta,R}
 \leq 2n c^{-1}(C_G+1)+\frac{B^2}{c(1-\theta)}
 \leq\frac{2n c^{-1}(C_G+1)+B^2/c}{1-\theta}.
\]
Thus the coefficient in this bound is independent of $R$ and $\theta$. Since
$vZ_R^\theta\uparrow vZ^\theta$, monotone convergence now proves \eqref{J theta bound}.

\paragraph{Step 3.}
To obtain the background gradient bound, we pair \eqref{J theta bound} with a second weighted energy
estimate on $\{Z<N\}$. For $N\geq2$ and $a\geq1$, we define
\[
 b_N(a)=\int_1^{\min\{a,N\}}\xi^{-\theta}\,d\xi
       =\frac{\min\{a,N\}^{1-\theta}-1}{1-\theta}.
\]
This function is increasing until $a=N$ and is then constant. More explicitly,
\[
 b_N'(a)=
 \begin{cases}
  a^{-\theta},&1<a<N,\\
  0,&a>N,
 \end{cases}
 \qquad
 b_N(\infty)=\frac{N^{1-\theta}-1}{1-\theta}.
\]
The composition $b_N(Z)$ is Lipschitz and constant near $x$. Its value at the pole is $b_N(\infty)$,
and
\[
 \nabla(b_N(Z))
 =\mathbf{1}_{\{Z<N\}}Z^{-\theta}\nabla Z
 \qquad\text{almost everywhere}.
\]
To apply the weak equation from Step~1, we use $b=b_N(Z)-b_N(\infty)$, which vanishes near $x$.
The weak identity then takes the following form, since $\nabla b=\nabla(b_N(Z))$ and
$\int_Md\mu_\varphi=1$.
\[
 -\int_{\{Z<N\}}Z^{-\theta}|\nabla Z|_\varphi^2\,d\mu_\varphi
 =-\int_M\langle\nabla Z,\nabla(b_N(Z))\rangle_\varphi
                                                \,d\mu_\varphi
 =\int_Mb_N(Z)\,d\mu_\varphi-b_N(\infty).
\]
Rearranging and using $b_N(Z)\geq0$, we obtain
 \begin{align}\label{Z theta bound}
 \int_{\{Z<N\}}Z^{-\theta}|\nabla Z|_\varphi^2\,d\mu_\varphi
 \leq b_N(\infty)
=\frac{N^{1-\theta}-1}{1-\theta}
 \leq\frac{N^{1-\theta}}{1-\theta}.
 \end{align}

To combine the two weighted estimates, first use the pointwise comparison
$$
|\nabla Z|_g
 \leq v^{1/2}|\nabla Z|_\varphi
 =(vZ^\theta)^{1/2}
            \bigl(Z^{-\theta}|\nabla Z|_\varphi^2\bigr)^{1/2}.
$$
Combining \eqref{J theta bound} with \eqref{Z theta bound}, we obtain
 \begin{align*}
 \int_M|\nabla Z_N|_g\,d\mu_g
 &=\int_{\{Z<N\}}|\nabla Z|_g\,d\mu_g
 \leq\left(\int_{\{Z<N\}}vZ^\theta\,d\mu_g\right)^{1/2}
       \left(\int_{\{Z<N\}}Z^{-\theta}|\nabla Z|_\varphi^2
                                                 \,d\mu_g\right)^{1/2}\\
 &\leq c^{-1/2}J_\theta^{1/2}
       \left(\int_{\{Z<N\}}Z^{-\theta}|\nabla Z|_\varphi^2
                                           \,d\mu_\varphi\right)^{1/2}
 \leq\frac{C N^{(1-\theta)/2}}{1-\theta}.
 \end{align*}
The constant $C$ in the last line is independent of $\theta\in(0,1)$, $N$, and $x$.

To obtain the logarithmic bound, take
$ \theta=1-\frac{1}{\log N} $ for $N\geq e^2$. This choice ensures $1/2\leq\theta<1$ and
\[
 N^{(1-\theta)/2}
 =\exp\!\left(\frac{1-\theta}{2}\log N\right)=e^{1/2},
 \qquad \frac{1}{1-\theta}=\log N.
\]
Hence
\[
 \int_M|\nabla Z_N|_g\,d\mu_g\leq C\log N, \quad N\geq e^2.
\]
For $2\leq N<e^2$, set $\theta=1/2$ in the same estimate to obtain
\[
 \int_M|\nabla Z_N|_g\,d\mu_g
 \leq2C N^{1/4}\leq2C e^{1/2}.
\]
Together, the two cases prove the first estimate in \eqref{eq:Green-truncations} after increasing
$C$.

\paragraph{Step 4.}
It remains to estimate the tail. Fix $1<\nu<n/(n-1)$. Using $Z\leq C_G+1+|G_x^\varphi|$ and
convexity, we obtain
 \begin{equation}\label{Z nu bound}
 \begin{split}
 \int_MZ^\nu\,d\mu_g
 &\leq2^{\nu-1}\left((C_G+1)^\nu
                   +\int_M|G_x^\varphi|^\nu\,d\mu_g\right)\\
 &\leq2^{\nu-1}\left((C_G+1)^\nu
                   +c^{-1}\int_M|G_x^\varphi|^\nu\,d\mu_\varphi\right)
 \leq C_\nu.
 \end{split}
 \end{equation}
The last inequality uses the $L^\nu$ estimate for the Green function \eqref{eq:Green-estimate}.

The function $Z-Z_N$ vanishes on $\{Z\leq N\}$ and equals $Z-N$ on $\{Z>N\}$. On the latter set,
since $1-\nu<0$, we have
\[
 Z^{1-\nu}\leq N^{1-\nu},
 \qquad Z=Z^\nu Z^{1-\nu}\leq N^{1-\nu}Z^\nu.
\]
By \eqref{Z nu bound}, we obtain
\[
 \begin{aligned}
 &\int_M(Z-Z_N)\,d\mu_g
 =\int_{\{Z>N\}}(Z-N)\,d\mu_g
 \leq\int_{\{Z>N\}}Z\,d\mu_g\\
 &\leq N^{1-\nu}\int_{\{Z>N\}}Z^\nu\,d\mu_g
 \leq N^{1-\nu}\int_MZ^\nu\,d\mu_g
 \leq C_\nu N^{1-\nu}.
 \end{aligned}
\]
This proves the second estimate in \eqref{eq:Green-truncations}.
\end{proof}

To use the truncation bounds in a heat approximation, we need to estimate differences averaged
against the heat kernel. The following lemma controls these averages by the background gradient
norm. We use the heat semigroup estimate of Ledoux \cite{Ledoux}, the coarea formula and the
Gaussian bounds recalled in Section~\ref{Heat estimates and interpolation}.

\begin{lem}
\label{lem:heat-W11}
Let $k_t(y,z)$ denote the heat kernel of $\Delta_g$ relative to $d\mu_g$. For every
$b\in W^{1,1}(M,g)$ and $0<t\leq1$, we have
\beq
 \iint_{M\times M}k_t(y,z)d_g(y,z)^\alpha
                 |b(y)-b(z)|\,d\mu_g(y)d\mu_g(z)
 \leq Ct^{(1+\alpha)/2}\|\Na b\|_{L^1(d\mu_g)}.
 \label{eq:heat-W11}
\eeq
For $b\in L^1(M,g)$, the same integral is bounded by $Ct^{\alpha/2}\|b\|_{L^1(d\mu_g)}$.
\end{lem}
\begin{proof}
We divide the proof into three steps and denote the double integral in \eqref{eq:heat-W11} by
$I_\alpha(t,b)$.

\paragraph{Step 1.}
Recall the following kernel properties from Section~\ref{Heat estimates and interpolation}.
\begin{align}\label{nabla Ps}
 \|\Na P_s h\|_\infty
 \leq Cs^{-1/2}\|h\|_\infty,
 \qquad h\in L^\infty(M),\quad 0<s\leq1.
\end{align}
For an open set $E$ with smooth boundary, let $\operatorname{Per}_g(E)$ be its perimeter with the
background volume normalization, and set
$$
 H_E(s)=\int_E P_s\mathbf{1}_{M\setminus E}\,d\mu_g.
$$
The heat semigroup estimate of Ledoux \cite[Section~1, inequalities~(6)--(7)]{Ledoux} yields
$$
 H_E(s)\leq C\sqrt{s}\,\operatorname{Per}_g(E),
 \qquad 0<s\leq1.
$$
The hypotheses hold because the smooth compact background has a lower Ricci curvature bound.
Our normalization of $\Delta_g$ and $d\mu_g$ affects only the constant.

For a real smooth function $b$, set $E_r=\{b>r\}$. Note that
$$
 |b(y)-b(z)|=\int_{\mathbb R}
 |\mathbf{1}_{E_r}(y)-\mathbf{1}_{E_r}(z)|\,dr.
$$
Almost every $r$ is a regular value of $b$. By symmetry of the heat kernel, Tonelli's theorem and
the coarea formula, we obtain
\begin{equation}\label{Ksb}
\begin{split}
 &\iint_{M\times M}k_s(y,z)|b(y)-b(z)|\,d\mu_g(y)d\mu_g(z)\\
 &=2\int_{\mathbb R}H_{E_r}(s)\,dr
 \leq C\sqrt{s}\int_{\mathbb R}\operatorname{Per}_g(E_r)\,dr
 \leq C\sqrt{s}\,\|\Na b\|_{L^1(d\mu_g)},
 \qquad 0<s\leq1.
\end{split}
\end{equation}

\paragraph{Step 2.}
We pass from \eqref{Ksb} to the weighted estimate by absorbing the distance factor into the
Gaussian. Let $C\geq1$ be a constant for which \eqref{Gaussian bounds} holds, and set $\kappa=2C^2$.
For $d=d_g(y,z)$ and $0<t\leq\kappa^{-1}$, combine the Gaussian upper bound with
$u^\alpha e^{-u^2/C}\leq C_{\alpha,C}e^{-u^2/(2C)}$ to obtain
$$
 d^\alpha k_t(y,z)
 \leq Ct^{\alpha/2}t^{-n}\exp\!\left(-\frac{d^2}{2Ct}\right).
$$
The lower bound in \eqref{Gaussian bounds}, applied at time $\kappa t\leq1$, therefore implies
\begin{align}\label{the comparison}
 d_g(y,z)^\alpha k_t(y,z)
 \leq Ct^{\alpha/2}k_{\kappa t}(y,z),
 \qquad 0<t\leq\kappa^{-1}.
\end{align}
Combining this comparison with \eqref{Ksb} at time $s=\kappa t$, we obtain
$$
 I_\alpha(t,b)
 \leq Ct^{\alpha/2}\sqrt{\kappa t}\,
             \|\Na b\|_{L^1(d\mu_g)}
 \leq Ct^{(1+\alpha)/2}\|\Na b\|_{L^1(d\mu_g)}.
$$
For $\kappa^{-1}\leq t\leq1$, we use the diameter bound together with \eqref{Ksb} at time $s=t$ to obtain
$$
 \begin{aligned}
 I_\alpha(t,b)
 &\leq C\operatorname{diam}_g(M)^\alpha\sqrt{t}\,
                    \|\Na b\|_{L^1(d\mu_g)}\\
 &\leq Ct^{(1+\alpha)/2}\|\Na b\|_{L^1(d\mu_g)},
 \end{aligned}
$$
where the last constant also absorbs $\kappa^{\alpha/2}$. This proves \eqref{eq:heat-W11} for smooth
$b$.

\paragraph{Step 3.}
To pass from smooth functions to $W^{1,1}$ functions, we first establish the integral bound for
$L^1$ functions. The moment estimate \eqref{kernel mass} with $\sigma=\alpha$ yields
\begin{align}\label{k d alpha}
 \int_M k_t(y,z)d_g(y,z)^\alpha\,d\mu_g(z)
 \leq Ct^{\alpha/2},\qquad 0<t\leq1.
\end{align}
Using symmetry and the triangle inequality, we obtain, for every $b\in L^1(M,g)$,
$$
 \begin{aligned}
 I_\alpha(t,b)
 &\leq2\int_M|b(y)|
       \left(\int_Mk_t(y,z)d_g(y,z)^\alpha\,d\mu_g(z)\right)d\mu_g(y)\\
 &\leq Ct^{\alpha/2}\|b\|_{L^1(d\mu_g)}.
 \end{aligned}
$$
For $b\in W^{1,1}(M,g)$, choose smooth $b_j$ converging to $b$ in $W^{1,1}$. For each fixed $t>0$,
the preceding bound implies
$$
 |I_\alpha(t,b_j)-I_\alpha(t,b)|
 \leq I_\alpha(t,b_j-b)
 \leq Ct^{\alpha/2}\|b_j-b\|_{L^1(d\mu_g)}\longrightarrow0.
$$
Since $\|\Na b_j\|_{L^1}\to\|\Na b\|_{L^1}$, by taking the limit we have \eqref{eq:heat-W11}. If $b$
is complex-valued, the same conclusion follows by applying the real-valued estimate to its real and
imaginary parts.
\end{proof}

We can now combine Lemmas~\ref{lem:truncated-Green} and \ref{lem:heat-W11} to estimate the Green
pairing with the heat approximation error. By symmetry of the heat kernel, the pairing can be
written as a double integral of the product of the two differences against the heat kernel.
To bound the contribution from the truncated Green function, we use the gradient estimate.

\begin{lem}
\label{lem:Green-cancellation}
Under the assumptions of Lemma~\ref{lem:truncated-Green}, for $a\in C^\alpha(M)$ and $0<t\leq1/4$,
we have
\beq
 \left|\int_MG_x^\varphi(a-P_ta)\,d\mu_g\right|
 \leq C\|a\|_{C^\alpha}
                  t^{(1+\alpha)/2}(1+|\log t|).
 \label{eq:Green-cancellation}
\eeq
\end{lem}
\begin{proof}
We divide the proof into two steps.

\paragraph{Step 1.}
We first write the Green pairing as an integral of differences. The exchange of variables is
justified by the following integrability bounds. Using the normalization and upper bound for
the Green function, we obtain
$$
 \begin{aligned}
 \|G_x^\varphi\|_{L^1(d\mu_g)}
 &\leq c^{-1}\int_M|G_x^\varphi|\,d\mu_\varphi\\
 &=2c^{-1}\int_M\max\{G_x^\varphi,0\}\,d\mu_\varphi
 \leq2c^{-1}C_G.
 \end{aligned}
$$
Consequently, the integral in \eqref{eq:Green-cancellation} is well-defined. Moreover, the mass
identities imply
$$
 \begin{aligned}
 &\iint_{M\times M}k_t(y,z)
   \bigl(|G_x^\varphi(y)|+|G_x^\varphi(z)|\bigr)
   |a(y)-a(z)|\,d\mu_g(y)d\mu_g(z)\\
 &\hspace{15mm}\leq4\|a\|_\infty\|G_x^\varphi\|_{L^1(d\mu_g)}<\infty.
 \end{aligned}
$$
By Fubini's theorem and the symmetry $k_t(y,z)=k_t(z,y)$, we obtain
$$
 \begin{aligned}
 I&=\int_MG_x^\varphi(a-P_ta)\,d\mu_g\\
 &=\frac12\iint_{M\times M}k_t(y,z)
   \bigl(G_x^\varphi(y)-G_x^\varphi(z)\bigr)
   \bigl(a(y)-a(z)\bigr)\,d\mu_g(y)d\mu_g(z).
 \end{aligned}
$$

\paragraph{Step 2.}
To estimate the Green function difference, we separate the truncation from its tail. For $N\geq2$,
let $Z$ and $Z_N$ be as in Lemma~\ref{lem:truncated-Green}. Then
$$
 G_x^\varphi=C_G+1-Z_N-(Z-Z_N).
$$
By the H\"older continuity of $a$, we have $|I|\leq\frac12[a]_{C^\alpha}(I_1+I_2)$, where
$$
 \begin{aligned}
 I_1&=\iint_{M\times M}k_t(y,z)d_g(y,z)^\alpha
             |Z_N(y)-Z_N(z)|\,d\mu_g(y)d\mu_g(z),\\
 I_2&=\iint_{M\times M}k_t(y,z)d_g(y,z)^\alpha
             |(Z-Z_N)(y)-(Z-Z_N)(z)|\,d\mu_g(y)d\mu_g(z).
 \end{aligned}
$$
Since $Z_N\in W^{1,1}(M,g)$, Lemmas~\ref{lem:truncated-Green} and \ref{lem:heat-W11} imply
$$
 I_1\leq Ct^{(1+\alpha)/2}\|\Na Z_N\|_{L^1(d\mu_g)}
       \leq Ct^{(1+\alpha)/2}(1+\log N).
$$
Combining the integral bound in Lemma~\ref{lem:heat-W11} with the tail estimate in
\eqref{eq:Green-truncations}, we obtain, for every fixed $1<\nu<n/(n-1)$,
$$
 I_2\leq Ct^{\alpha/2}\|Z-Z_N\|_{L^1(d\mu_g)}
       \leq C_\nu t^{\alpha/2}N^{1-\nu}.
$$
Combining these inequalities, we obtain
$$
 |I|\leq C_\nu\|a\|_{C^\alpha}
       \left\{t^{(1+\alpha)/2}(1+\log N)
                    +t^{\alpha/2}N^{1-\nu}\right\}.
$$
To make the tail term of order $t^{(1+\alpha)/2}$, we choose
$$
 \nu=1+\frac1{2(n-1)},\qquad
 N=t^{-1/(2(\nu-1))}=t^{-(n-1)}.
$$
Since $n\geq2$ and $0<t\leq1/4$, these choices are admissible. For these values, we have
$$
 t^{\alpha/2}N^{1-\nu}=t^{(1+\alpha)/2},\qquad
 \log N=(n-1)|\log t|.
$$
Combining the above estimates, we have \eqref{eq:Green-cancellation}.
\end{proof}

We now apply these estimates to the smoothed densities. Lemma~\ref{lem:elliptic-trace} supplies
the trace bound in Proposition~\ref{prop:density-approximation}, while
Lemma~\ref{lem:Green-cancellation} controls the variation of the solution along a path of densities.

\begin{proof}[Proof of Proposition~\ref{prop:density-approximation}] We divide the proof into five steps.

\paragraph{Step 1.}
We first construct the smooth approximations and establish their trace bound. For
$0<\varepsilon\leq1/2$, we recall that
\[
 f_\varepsilon=P_{\varepsilon^2}(e^F),\qquad
 F_\varepsilon=\log f_\varepsilon.
\]
Under the assumptions on $F$, the exponential satisfies
$ e^{-K}\leq e^F\leq e^K, [e^F]_{C^\alpha}\leq e^K[F]_{C^\alpha}, \|e^F\|_{C^\alpha}\leq e^K(1+K). $
Since $P_t$ preserves positivity, constants, and the $d\mu_g$-integral, we have
\[
 e^{-K}\leq f_\varepsilon\leq e^K,\qquad
 \int_M f_\varepsilon\,d\mu_g
 =\int_M e^F\,d\mu_g=1.
\]
Thus each $f_\varepsilon$ is a smooth positive density satisfying the compatibility condition for
the elliptic equation. By Yau \cite{Yau}, there is a unique smooth solution with the normalization
\[
 (\omega_g+\sqrt{-1}\partial\bar\partial\varphi_\varepsilon)^n
 =f_\varepsilon\omega_g^n,\qquad
 \omega_{\varphi_\varepsilon}>0,\qquad
 \int_M\varphi_\varepsilon\,d\mu_g=0.
\]
By the heat estimates \eqref{eq:heat-alpha} applied to $e^F$, we have
\begin{align}\label{heat estimates app}
 \|f_\varepsilon\|_{C^\alpha}\leq C,\qquad
 \|\nabla f_\varepsilon\|_\infty\leq C\varepsilon^{\alpha-1},\qquad
 \|f_\varepsilon-e^F\|_\infty\leq C\varepsilon^\alpha.
\end{align}
Consequently, we have
\[
 \|F_\varepsilon\|_\infty\leq K,\qquad
 [F_\varepsilon]_{C^\alpha}\leq e^K[f_\varepsilon]_{C^\alpha},
 \qquad \|F_\varepsilon\|_{C^\alpha}\leq C.
\]
By the chain rule,
\[
 \nabla F_\varepsilon
 =f_\varepsilon^{-1}\nabla f_\varepsilon,\qquad
 1+\|\nabla F_\varepsilon\|_\infty
 \leq C\varepsilon^{\alpha-1}.
\]
Applying Lemma~\ref{lem:C0} and the gradient estimate \eqref{eq:gradient-estimate} with logarithmic
density $F_\varepsilon$, we obtain
\begin{align}\label{Linfty gradient estimate app}
 \|\varphi_\varepsilon\|_\infty
 +\|\nabla\varphi_\varepsilon\|_\infty\leq C.
\end{align}
Lemma~\ref{lem:elliptic-trace} now yields the following estimate for every fixed $\delta>0$.
\begin{align}\label{trace estimate app}
 \|n+\Delta_g\varphi_\varepsilon\|_\infty
 &\leq C_\delta
       \bigl(1+\|\nabla F_\varepsilon\|_\infty\bigr)^{1+\delta}
 \leq C_\delta\varepsilon^{(\alpha-1)(1+\delta)}
 =C_\delta\varepsilon^{-(1-\alpha)(1+\delta)}.
\end{align}
This proves the trace assertion of \eqref{eq:density-approximation}.

\paragraph{Step 2.}
To compare $\varphi_\varepsilon$ and $\varphi_{2\varepsilon}$, we join their densities by an affine
path. Fix $0<\varepsilon\leq1/4$ and set
$ f_{\varepsilon,s}=(1-s)f_\varepsilon+s f_{2\varepsilon}$, $0\leq s\leq1$. By convexity and the bounds on the two endpoint densities, we have
\begin{align}\label{approximate density bound}
 e^{-K}\leq f_{\varepsilon,s}\leq e^K,\qquad
 \int_M f_{\varepsilon,s}\,d\mu_g=1,\qquad
 \|f_{\varepsilon,s}\|_{C^\alpha}\leq C,\qquad
 \|\log f_{\varepsilon,s}\|_{C^\alpha}\leq C.
\end{align}
For each $s$, we denote by $\varphi_{\varepsilon,s}$ the smooth normalized solution of the
approximate equation
\[
 (\omega_g+\sqrt{-1}\partial\bar\partial
                  \varphi_{\varepsilon,s})^n
 =f_{\varepsilon,s}\omega_g^n,\qquad
 \omega_{\varphi_{\varepsilon,s}}>0,\qquad
 \int_M\varphi_{\varepsilon,s}\,d\mu_g=0.
\]
By uniqueness, the endpoint potentials satisfy $ \varphi_{\varepsilon,0}=\varphi_\varepsilon$ and
$\varphi_{\varepsilon,1}=\varphi_{2\varepsilon}. $ Applying Lemma~\ref{lem:C0} and
\eqref{eq:gradient-estimate} once more, we obtain the following bound with $C$ independent of
$\varepsilon$.
\begin{align}\label{approximate gradient bound}
 \sup_{0\leq s\leq1}
 \bigl(\|\varphi_{\varepsilon,s}\|_\infty
      +\|\nabla\varphi_{\varepsilon,s}\|_\infty\bigr)\leq C.
\end{align}

To justify differentiation in $s$, we fix $\varepsilon>0$ and a point $s_0\in[0,1]$. The
differential of the complex Monge--Amp\`ere operator at $\varphi_{\varepsilon,s_0}$ is
\[
 \xi\longmapsto
 f_{\varepsilon,s_0}
 \Delta_{\varphi_{\varepsilon,s_0}}\xi.
\]
For an integer $k\geq0$, $0<\eta<1$ and any $h\in C^{k,\eta}$ with $\int_Mh\,d\mu_g=0$, the equation
for the inverse is
\[
 \Delta_{\varphi_{\varepsilon,s_0}}\xi
 =\frac{h}{f_{\varepsilon,s_0}},\qquad
 \int_M\xi\,d\mu_g=0.
\]
Its solvability condition is exactly
$ \int_M\frac{h}{f_{\varepsilon,s_0}} \,d\mu_{\varphi_{\varepsilon,s_0}} =\int_M h\,d\mu_g=0, $
since $d\mu_{\varphi_{\varepsilon,s_0}} =f_{\varepsilon,s_0}\,d\mu_g$. By elliptic solvability and Schauder regularity, as in Gilbarg--Trudinger
\cite[Chapter~6]{GilTru}, the differential is an isomorphism from the mean-zero subspace of
$C^{k+2,\eta}$ to that of $C^{k,\eta}$, with both means taken relative to $d\mu_g$. These are the
tangent spaces for the density map between functions with mean zero and densities with integral
one.
The implicit function theorem yields a smooth local family in $s$. By uniqueness, the local
families agree on overlaps, so the dependence on $s$ is smooth on the whole interval. The same argument applies at the endpoints by locally extending the positive affine
density path.

Differentiating the approximate equation, we obtain
\[
\begin{aligned}
 \Delta_{\varphi_{\varepsilon,s}}
                   (\partial_s\varphi_{\varepsilon,s})
 =\partial_s\log
   \frac{\omega_{\varphi_{\varepsilon,s}}^n}{\omega_g^n}
 =\partial_s\log f_{\varepsilon,s}
 =\frac{f_{2\varepsilon}-f_\varepsilon}{f_{\varepsilon,s}},\quad  \int_M\partial_s\varphi_{\varepsilon,s}\,d\mu_g=0.
\end{aligned}
\]
For each fixed $s$, we apply Green's formula with the normalization
\eqref{eq:Green-normalization} to obtain
\begin{equation}\label{Green formula derivative}
\begin{split}
 &\partial_s\varphi_{\varepsilon,s}(x)
   -\int_M\partial_s\varphi_{\varepsilon,s}
                         \,d\mu_{\varphi_{\varepsilon,s}}
=
 \int_MG_x^{\varphi_{\varepsilon,s}}
 \Delta_{\varphi_{\varepsilon,s}}
                 (\partial_s\varphi_{\varepsilon,s})
                         \,d\mu_{\varphi_{\varepsilon,s}}\\
&=
 \int_MG_x^{\varphi_{\varepsilon,s}}
 \frac{f_{2\varepsilon}-f_\varepsilon}{f_{\varepsilon,s}}
                         \,d\mu_{\varphi_{\varepsilon,s}}
=
 \int_MG_x^{\varphi_{\varepsilon,s}}
                  (f_{2\varepsilon}-f_\varepsilon)\,d\mu_g= R_{\varepsilon,s}(x).
\end{split}
\end{equation}

\paragraph{Step 3.}
We now estimate the right-hand side of \eqref{Green formula derivative} using
Lemma~\ref{lem:Green-cancellation}. By the semigroup property of the fixed background heat operator,
we have
\[
\begin{aligned}
 f_{2\varepsilon}
 =P_{(2\varepsilon)^2}(e^F)
 =P_{4\varepsilon^2}(e^F)
 =P_{3\varepsilon^2}P_{\varepsilon^2}(e^F)
 =P_{3\varepsilon^2}f_\varepsilon.
\end{aligned}
\]
It follows that $f_{2\varepsilon}-f_\varepsilon=-(f_\varepsilon-P_{3\varepsilon^2}f_\varepsilon)$.
Lemma~\ref{lem:Green-cancellation} then applies with the smooth potential $\varphi_{\varepsilon,s}$,
the function $a=f_\varepsilon$, and $t=3\varepsilon^2\leq\frac{3}{16}<\frac14. $ Combining the
density bounds \eqref{approximate density bound} and the gradient estimate
\eqref{approximate gradient bound} with $\|f_\varepsilon\|_{C^\alpha}\leq C$, we obtain
\[
\begin{aligned}
 \|R_{\varepsilon,s}\|_\infty
 &\leq C\|f_\varepsilon\|_{C^\alpha}
 (3\varepsilon^2)^{(1+\alpha)/2}
       (1+|\log(3\varepsilon^2)|)
 \leq C\varepsilon^{1+\alpha}(1+|\log\varepsilon|),
 \qquad 0\leq s\leq1.
\end{aligned}
\]
The constants are uniform in $s$ and $\varepsilon$.

To recover the normalized derivative, integrate its Green representation
\eqref{Green formula derivative} with respect to $d\mu_g(x)$. Since the derivative has mean
zero with respect to $d\mu_g$, we obtain
$ -\int_M\partial_s\varphi_{\varepsilon,s} \,d\mu_{\varphi_{\varepsilon,s}} =\int_MR_{\varepsilon,s}\,d\mu_g. $
The Green representation can thus be written as
\[
 \partial_s\varphi_{\varepsilon,s}(x)
 =R_{\varepsilon,s}(x)
       -\int_MR_{\varepsilon,s}\,d\mu_g.
\]
Therefore,
\[
 \|\partial_s\varphi_{\varepsilon,s}\|_\infty
 \leq2\|R_{\varepsilon,s}\|_\infty
 \leq C\varepsilon^{1+\alpha}(1+|\log\varepsilon|).
\]
Integrating from $s=0$ to $s=1$, we conclude that
\[
 \|\varphi_{2\varepsilon}-\varphi_\varepsilon\|_\infty
 \leq\int_0^1
       \|\partial_s\varphi_{\varepsilon,s}\|_\infty\,ds
 \leq C\varepsilon^{1+\alpha}(1+|\log\varepsilon|).
\]

\paragraph{Step 4.}
The estimate between successive scales allows us to pass to the limit along a dyadic sequence. Fix $0<\varepsilon\leq1/4$ and
set $\varepsilon_j=\varepsilon2^{-j}$ for $j=0,1,2,\ldots$. For $j\geq1$, it follows from the
preceding comparison at scale $\varepsilon_j$, where $2\varepsilon_j=\varepsilon_{j-1}$, that
\[
 \|\varphi_{\varepsilon_j}
          -\varphi_{\varepsilon_{j-1}}\|_\infty
 \leq C\varepsilon^{1+\alpha}2^{-j(1+\alpha)}
      (1+|\log\varepsilon|+j\log2)
  \leq C(1+|\log\varepsilon|)\varepsilon^{1+\alpha}(1+j)2^{-j(1+\alpha)}.
\]
Since $ \sum_{j=1}^{\infty}(1+j)2^{-j(1+\alpha)} \leq C, $ the sequence $\varphi_{\varepsilon_j}$ is
Cauchy in $C^0(M)$. Its uniform limit $\widetilde\varphi$ is bounded, continuous and
$\omega_g$-plurisubharmonic, with $\int_M\widetilde\varphi\,d\mu_g=0$.

We next identify this limit with the prescribed weak solution. The heat approximation
\eqref{heat estimates app} implies
\[
 \|f_{\varepsilon_j}-e^F\|_\infty
 \leq C\varepsilon_j^\alpha\longrightarrow0.
\]
After adding the local background potentials, we apply the continuity of the Monge--Amp\`ere
operator under locally uniform convergence of bounded plurisubharmonic functions, due to
Bedford--Taylor \cite{BT1,BT2}, to obtain
\[
 (\omega_g+\sqrt{-1}\partial\bar\partial
                      \widetilde\varphi)^n
 =e^F\omega_g^n.
\]
Since $e^F$ is bounded and has the required total mass, the uniqueness theorem of
Eyssidieux--Guedj--Zeriahi \cite[Theorem~A]{EGZ2009} implies that the two bounded
pluripotential solutions differ by a constant. They have the same $d\mu_g$-mean, so
$\widetilde\varphi=\varphi$.

Summing the consecutive-scale differences and using $\widetilde\varphi=\varphi$, we obtain the
following estimate for every $0<\varepsilon\leq1/4$
\begin{align}\label{L infty app}
 \|\varphi_\varepsilon-\varphi\|_\infty
 \leq C\varepsilon^{1+\alpha}(1+|\log\varepsilon|)
  \leq C_\rho\varepsilon^{1+\alpha-\rho}.
\end{align}
For the second inequality, fix $\rho>0$ and use
$ \varepsilon^\rho(1+|\log \varepsilon|) =e^{-\rho |\log \varepsilon|}(1+|\log \varepsilon|) \leq1+\frac{1}{e\rho}. $
This also proves that the full family converges uniformly to $\varphi$ as $\varepsilon\downarrow0$,
and establishes continuity of the originally prescribed bounded solution.

\paragraph{Step 5.}
It remains to combine the approximation estimate from Step~4 with the trace bound from Step~1.
Fix $\delta>0$ and $0<\gamma<1$, and choose a finite exponent $p>2n/(1-\gamma)$. Since
$\int_Md\mu_g=1$, the trace estimate \eqref{trace estimate app} implies
\[
\begin{aligned}
 \|\Delta_g\varphi_\varepsilon\|_{L^p(d\mu_g)}
 &\leq\|\Delta_g\varphi_\varepsilon\|_\infty
 \leq n+\|n+\Delta_g\varphi_\varepsilon\|_\infty
 \leq C_\delta\varepsilon^{-(1-\alpha)(1+\delta)}.
\end{aligned}
\]
It follows from \eqref{eq:background-elliptic} and \eqref{Linfty gradient estimate app} that
\[
\begin{aligned}
 \|\varphi_\varepsilon\|_{C^{1,\gamma}}
 \leq C_{p,\gamma}\|\varphi_\varepsilon\|_{W^{2,p}}
 \leq C_{p,\gamma}
 \bigl(\|\Delta_g\varphi_\varepsilon\|_{L^p(d\mu_g)}
                    +\|\varphi_\varepsilon\|_\infty\bigr)
 \leq C_{\gamma,\delta}
             \varepsilon^{-(1-\alpha)(1+\delta)}.
\end{aligned}
\]
Now fix $0<\beta<\alpha$ and set $ a=1+\alpha-\rho, b=(1-\alpha)(1+\delta) $. To meet the
condition \eqref{eq:interpolation-range} in Lemma~\ref{lem:interpolation}, choose
\begin{align}\label{parameters}
 \rho=\frac{\alpha-\beta}{2(1-\beta)},\qquad
 \delta=\frac{\alpha-\beta}{2(1-\alpha)(1+\beta)},\qquad \gamma=\frac12\left(1+\beta+\frac{b(1+\beta)}{a}\right).
\end{align}
With these choices, $0<\rho<1/2$, $\delta>0$ and $a>0$. Moreover,
$a(1-\beta)-b(1+\beta)=\alpha-\beta>0$, so that $\beta<\gamma<1$. Substituting the value of
$\gamma$, we obtain
$ a(\gamma-\beta)-b(1+\beta) =\frac12\bigl(a(1-\beta)-b(1+\beta)\bigr) =\frac{\alpha-\beta}{2}>0. $
Thus \eqref{eq:interpolation-range} holds. With \eqref{L infty app} and the preceding bound for
$\|\varphi_\varepsilon\|_{C^{1,\gamma}}$, Lemma~\ref{lem:interpolation} yields
\[
 \|\varphi\|_{C^{1,\beta}}\leq C_\beta.
\]
This proves the last assertion and completes the proof of the proposition.
\end{proof}

\subsection{An auxiliary parabolic equation}
\label{sec:flow}

In this subsection, we construct an auxiliary parabolic flow with forcing obtained by smoothing
$F$ with the background heat semigroup. The first estimates control the potential, its time
derivative and the difference between solutions with different initial potentials. With the density
and trace bounds, we continue the flow up to time one. Finally, by comparison and interpolation, we obtain a uniform H\"older estimate along the flow
from the initial estimates for the elliptic approximations.

Throughout this subsection, $F$ and the solution $\varphi$ of \eqref{eq:MA} are smooth.
Set $\widehat F(t)=P_tF$ and write $\Dpar=\partial_t-\Delta_\psi$ for the parabolic operator
along the flow. By the definition of $P_t$,
\[
 (\partial_t-\Delta_g)\widehat F=0,\qquad \widehat F(0)=F.
\]
For a smooth initial K\"ahler potential $z$, we denote by $\mathcal T_t^Fz=\psi(t)$ the
solution to
\beq
 \partial_t\psi
 =\log\frac{\omega_{\psi(t)}^n}{\omega_g^n}-\widehat F(t),
 \qquad \psi(0)=z.
 \label{eq:flow}
\eeq
The forcing $P_tF$ is the same for every initial potential $z$, so it cancels when we compare
solutions. The construction follows the parabolic regularization of Sz\'ekelyhidi--Tosatti
\cite{ST}. For related constructions with weak initial data, see  \cite{MR2348984,MR2775789,MR3680377,MR3570152, MR3595934} and references therein. In the quantitative estimates below, the dependence
on $F$ is only through its prescribed $C^\alpha$ bound.

\begin{lem}
\label{lem:flow-basic}
Assume that the smooth initial K\"ahler potential in \eqref{eq:flow} satisfies
$$
 \|z\|_\infty+
 \left\|\log\frac{\omega_z^n}{\omega_g^n}\right\|_\infty\leq B.
$$
Then there is a unique smooth solution to \eqref{eq:flow} on $[0,1]$. With
$U=\partial_t\psi$ and $H=\log(\omega_{\psi(t)}^n/\omega_g^n)$, we have
\beq
 \|\psi(t)\|_\infty+\|U(t)\|_\infty+\|H(t)\|_\infty\leq C_B.
 \label{eq:flow-basic}
\eeq
For any two smooth initial K\"ahler potentials, the corresponding solutions satisfy
\beq
 \|\mathcal T_t^Fz_1-\mathcal T_t^Fz_2\|_\infty
 \leq\|z_1-z_2\|_\infty.
 \label{eq:flow-contraction}
\eeq
For $z=\varphi$, the stronger estimates are
\beq
 \|U(t)\|_\infty\leq Ct^{\alpha/2},\qquad
 \|\psi(t)-\varphi\|_\infty\leq Ct^{1+\alpha/2}.
 \label{eq:flow-approximation}
\eeq
The constants depend only on $M,g,n,\alpha,K$ and, in \eqref{eq:flow-basic}, on $B$.
\end{lem}

\begin{proof}
We first prove the estimates as long as the smooth solution exists. The continuation argument
follows Lemma~\ref{lem:trace-propagation}.

\paragraph{Step 1.}
At the initial potential $z$, the linearization in the potential is $\Delta_z$. Since
$\omega_z>0$ and the forcing $P_tF$ is smooth up to $t=0$, short-time existence follows from
the strictly parabolic theory of Lunardi \cite[Sections~8.1 and 8.5.3]{Lunardi}. Denote the
maximal existence time by $T_*>0$ and fix $T<\min\{T_*,1\}$. The maximum principle applies on
$[0,T]$, where the solution is smooth and the metrics are positive.

Using $\partial_t(g_\psi)_{i\bar j}=U_{i\bar j}$, we differentiate the equation and the
volume form to obtain
\beq
 \Dpar U=-\partial_t\widehat F,\qquad
 \partial_tH=\Delta_\psi U,\qquad
 \partial_t d\mu_{\psi(t)}=(\Delta_\psi U)d\mu_{\psi(t)}.
 \label{eq:flow-identities}
\eeq
By the heat estimate \eqref{eq:heat-time},
\begin{align}\label{heat estimate V}
 \|\partial_t\widehat F(t)\|_\infty\leq Ct^{\alpha/2-1},
 \qquad 0<t\leq1.
\end{align}
For $0<\eta<t\leq T$, set
$$
 V(x,s)=U(x,s)-\|U(\eta)\|_\infty
              -C\int_\eta^s r^{\alpha/2-1}\,dr,
 \qquad \eta\leq s\leq t.
$$
Then $(\partial_s-\Delta_{\psi(s)})V\leq0$ and $V(\eta)\leq0$. Applying the maximum
principle to $V$ and then using the same argument for $-U$, we obtain
\begin{align}\label{U t bound}
 \|U(t)\|_\infty
 &\leq\|U(\eta)\|_\infty+C\int_\eta^t r^{\alpha/2-1}\,dr
 \leq\|U(\eta)\|_\infty+\frac{2C}{\alpha}t^{\alpha/2}.
\end{align}
Since $\alpha>0$, the forcing bound is integrable at zero. By smoothness at $t=0$, we may
let $\eta\downarrow0$ in the preceding estimate. Using $U(0)=\log f_z-F$, the resulting
bound is
\[
 \|U(t)\|_\infty\leq B+K+\frac{2C}{\alpha}t^{\alpha/2}.
\]
Integrating $\partial_t\psi=U$ and using $\|z\|_\infty\leq B$, we obtain
\[
 \|\psi(t)\|_\infty
 \leq B+(B+K)t+\frac{2C}{\alpha(1+\alpha/2)}t^{1+\alpha/2}.
\]
Moreover, $\|\widehat F(t)\|_\infty\leq K$ by the heat semigroup contraction. Since
$H=U+\widehat F$, this proves \eqref{eq:flow-basic} on the interval of existence.
After increasing $C_B$ if necessary, the density satisfies
\begin{equation}
 C_B^{-1}\leq f_{\psi(t)}=e^{H(t)}\leq C_B.
 \label{eq:flow-density}
\end{equation}
These constants are independent of $T$ and of the initial trace.

\paragraph{Step 2.}
The common forcing cancels in the difference of two solutions. Thus, on any common
interval of existence,
\[
 \partial_t(\psi_1-\psi_2)
 =\log\det g_{\psi_1}-\log\det g_{\psi_2}.
\]
Fix such an interval $[0,T]$, set $m=\sup_M(z_1-z_2)$ and choose $\eta>0$. If
$\tilde\psi=\psi_1-\psi_2-m-\eta t$ had a positive maximum at $(x_0,t_0)$ on
$M\times[0,T]$, then $t_0>0$ and
\begin{align}\label{maximum tilde psi}
 \partial_t\tilde\psi(x_0,t_0)\geq0.
\end{align}
At this point the complex Hessian of $\tilde\psi$ is nonpositive, so
$g_{\psi_1}\leq g_{\psi_2}$. By monotonicity of the determinant on positive Hermitian matrices,
$\partial_t\tilde\psi\leq-\eta<0$, a contradiction. Letting
$\eta\downarrow0$ and interchanging the two solutions, we obtain
\[
 \inf_M(z_1-z_2)\leq\psi_1(t)-\psi_2(t)\leq\sup_M(z_1-z_2).
\]
This proves \eqref{eq:flow-contraction} and uniqueness for a fixed initial value.

\paragraph{Step 3.}
When $z=\varphi$, we have $U(0)=\log f_\varphi-F=0$ by \eqref{eq:MA}. The bound in Step~1
then improves to $\|U(t)\|_\infty\leq Ct^{\alpha/2}$. Integrating in time from $0$ to $t$,
we obtain
\[
 \|\psi(t)-\varphi\|_\infty
 \leq C\int_0^t s^{\alpha/2}\,ds
 =\frac{C}{1+\alpha/2}t^{1+\alpha/2}.
\]
This proves \eqref{eq:flow-approximation}. All estimates hold for every
$T<\min\{T_*,1\}$, with constants independent of $T$.
\end{proof}

To extend the solution, we next propagate the initial trace bound. For this purpose,
consider a weighted trace whose evolution inequality has a coefficient integrable at the
initial time. This inequality will also be used in the energy estimates.

\begin{lem}
\label{lem:trace-propagation}
Under the assumptions of Lemma~\ref{lem:flow-basic}, we have
\beq
 \|n+\Delta_g\mathcal T_t^Fz\|_\infty
 \leq C_B\bigl(1+\|n+\Delta_gz\|_\infty\bigr),
 \qquad 0\leq t\leq1.
 \label{eq:trace-propagation}
\eeq
For $A>C_g+1$ fixed in terms of the background, the function $w=e^{-A\psi}v$, where
$v=n+\Delta_g\psi$, satisfies
\beq
 \Dpar w\leq L(t)w-w|\Na\log w|_\psi^2,\qquad
 L(t)=C_B(1+t^{\alpha/2-1}),\qquad 0<t\leq1.
 \label{eq:parabolic-W}
\eeq
\end{lem}

\begin{proof}
As above, fix $T<\min\{T_*,1\}$ and work on $[0,T]$. We first derive the evolution inequality
and then apply the maximum principle.

\paragraph{Step 1.}
Set $q=\tr_\psi g$. Combining $\partial_tv=\Delta_g(H-\widehat F)$ with
\eqref{eq:trace-differential}, we obtain
$$
 \Dpar\log v\leq C_gq-\frac{\Delta_g\widehat F}{v}.
$$
Since $\Dpar\psi=U-n+q$, the estimate for the weighted trace takes the following form by
the choice of $A$ and \eqref{eq:flow-basic}.
$$
 \begin{aligned}
 \Dpar\log w
 &\leq-A(U-n+q)+C_gq-\frac{\Delta_g\widehat F}{v}\\
 &\leq-q+C_B-\frac{\Delta_g\widehat F}{v}.
 \end{aligned}
$$
Using $\Dpar w=w\Dpar\log w-w|\Na\log w|_\psi^2$ and
$|\Delta_g\widehat F|\leq Ct^{\alpha/2-1}$, we obtain
$$
 \Dpar w\leq-wq+C_Bw+C_Bt^{\alpha/2-1}-w|\Na\log w|_\psi^2.
$$
Here the factor $w/v=e^{-A\psi}$ is uniformly bounded by \eqref{eq:flow-basic}.

To absorb the additive forcing term, we first bound the trace from below. By the
arithmetic--geometric mean inequality and \eqref{eq:flow-density},
$$
 v\geq n\exp(H/n)\geq c_B,\qquad w=e^{-A\psi}v\geq c_B.
$$
Here $c_B>0$ has the same dependence as $C_B$. Hence
$C_Bt^{\alpha/2-1}\leq C_Bc_B^{-1}t^{\alpha/2-1}w$. Discarding $-wq$ and
increasing $C_B$, we obtain \eqref{eq:parabolic-W} on $(0,T]$.

\paragraph{Step 2.}
For $0<\eta<t\leq T$, we apply the maximum principle on $[\eta,t]$ to
$\exp(-\int_\eta^\tau L(s)\,ds)w(\tau)$, which is a subsolution by \eqref{eq:parabolic-W}.
It follows that
$$
 \sup_Mw(t)\leq\exp\!\left(\int_\eta^tL(s)\,ds\right)\sup_Mw(\eta).
$$
Since $L$ is integrable at zero and $w$ is continuous there, we may let
$\eta\downarrow0$. In particular,
$$
 \int_0^tL(s)\,ds
 \leq C_B\left(t+\frac2\alpha t^{\alpha/2}\right)\leq C_B,
 \qquad 0\leq t\leq1.
$$
At time zero, $w(0)=e^{-Az}(n+\Delta_gz)$, and $\|z\|_\infty\leq B$. Since
$C_B^{-1}v\leq w\leq C_Bv$, the preceding comparison implies
\eqref{eq:trace-propagation} on the interval of existence. The constants are independent
of $T$ and of the initial trace.
\end{proof}

With the density and trace bounds, the equation is uniformly parabolic for each fixed smooth
initial potential. This allows us to use interior regularity and continue the solution.
The higher derivative bounds needed for continuation may depend on the fixed smooth data.

\begin{proof}[Completion of the existence assertion in Lemma~\ref{lem:flow-basic}]
Suppose that $T_*\leq1$. We divide the continuation argument into three steps.

\paragraph{Step 1.}
Applying Lemmas~\ref{lem:flow-basic} and \ref{lem:trace-propagation} before $T_*$, we obtain
\[
 \sup_{0\leq t<T_*}
 \bigl(\|\psi(t)\|_\infty+\|U(t)\|_\infty\bigr)\leq C_B,
 \qquad \sup_{0\leq t<T_*}\|v(t)\|_\infty\leq V_z,
\]
where $V_z=C_B(1+\|n+\Delta_gz\|_\infty)$. If $\lambda_i$ are the eigenvalues of
$g_\psi$ with respect to $g$, then
\[
 \lambda_i\leq V_z,\qquad
 \lambda_i=\frac{e^H}{\prod_{j\ne i}\lambda_j}
             \geq C_B^{-1}V_z^{-(n-1)}.
\]
Consequently, $c_zg\leq g_{\psi(t)}\leq C_zg$ for fixed $c_z,C_z>0$ independent of $t$.
These metric constants may depend on the initial trace.

On a fixed holomorphic coordinate ball, we choose a smooth real local potential $\rho$ for
$g$ and set $\Phi=\rho+\psi$. Then
\[
 \partial_t\Phi-\log\det(\Phi_{i\bar j})
 =-\log\det(g_{i\bar j})-P_tF,\qquad
 \lambda I\leq(\Phi_{i\bar j})\leq\Lambda I,
\]
for fixed positive constants $\lambda,\Lambda$. The logarithmic determinant is concave,
and the linearization, whose coefficients are the entries of the inverse matrix, is uniformly
elliptic on this set of Hermitian matrices. The complex Hessian is a fixed linear projection of the real Hessian
in these coordinates. The right-hand side is smooth with bounded derivatives for
$t\geq T_*/4$.
For the complex Monge--Amp\`ere equation, we can therefore apply the theorem of Chu
\cite[Theorem~5.1]{Chu} to obtain an interior parabolic $C^{2+\gamma,1+\gamma/2}$ bound on
smaller backward cylinders, for some $0<\gamma<1$. For $T_*/2\leq t<T_*$, the cylinders can be chosen with fixed spatial
radius and time length and with lower time boundary above $T_*/4$. The estimate is therefore
uniform as $t\uparrow T_*$.

\paragraph{Step 2.}
The preceding estimate makes $g_\psi^{i\bar j}$ uniformly H\"older continuous in space and
time. Differentiating the local equation in a real coordinate direction $\partial_a$, we obtain
\[
 \bigl(\partial_t-g_\psi^{i\bar j}\partial_i\partial_{\bar j}\bigr)(\partial_a\Phi)
 =-\partial_a\log\det(g_{i\bar j})-\partial_a(P_tF).
\]
The time derivative satisfies $(\partial_t-\Delta_\psi)U=-\partial_t(P_tF)$. These linear equations satisfy the hypotheses of the interior parabolic Schauder estimates
in Lunardi \cite[Chapter~5]{Lunardi}. By iterating the estimates on nested cylinders and using a finite
spatial cover, we obtain, for every integer $k\geq0$,
\[
 \sup_{T_*/2\leq t<T_*}
 \bigl(\|\psi(t)\|_{C^{k+2}(M,g)}+\|U(t)\|_{C^k(M,g)}\bigr)\leq C_k.
\]
Here $C_k$ may depend on $T_*$, the fixed smooth data and the initial trace, but is
independent of $t<T_*$. A related parabolic bootstrap appears in Sz\'ekelyhidi--Tosatti
\cite[Section~2]{ST}.

For $T_*/2\leq s<t<T_*$, we consequently have
\[
 \|\psi(t)-\psi(s)\|_{C^k(M,g)}
 \leq\int_s^t\|U(r)\|_{C^k(M,g)}\,dr\leq C_k|t-s|.
\]
Thus $\psi(t)$ converges in every spatial $C^k$ norm to a smooth potential $\psi_*$ as
$t\uparrow T_*$. The lower metric bound passes to the limit, and hence
$\omega_{\psi_*}\geq c_z\omega_g>0$.

\paragraph{Step 3.}
We apply the short-time existence theorem at time $T_*$ to
\[
 \partial_t\widetilde\psi
 =\log\frac{\omega_{\widetilde\psi(t)}^n}{\omega_g^n}-P_tF,
 \qquad \widetilde\psi(T_*)=\psi_*.
\]
The forcing is smooth through $T_*$. The equation determines the first time derivative
from the spatial derivatives of the potential and the forcing. By differentiating the
equation repeatedly in time, we can express each higher time derivative in terms of
spatial derivatives of lower time derivatives and derivatives of the forcing. The smooth
spatial convergence from Step~2 then shows inductively that the time derivatives agree
at the joining time. The two solutions therefore form a smooth continuation beyond $T_*$,
contradicting maximality. Hence $T_*>1$.

By continuity, all estimates proved before $T_*$ hold on the closed interval $[0,1]$.
The constants used only for continuation do not change the quantitative dependence of
\eqref{eq:flow-basic}, \eqref{eq:flow-approximation} or \eqref{eq:trace-propagation}.
\end{proof}

It remains to propagate the initial H\"older regularity from
Proposition~\ref{prop:density-approximation}. We evolve the elliptic approximations with the same forcing $P_tF$. The comparison principle
keeps the approximation error bounded by its initial value, and the preceding trace estimate
controls the spatial norms of the evolving approximations. These estimates allow us to apply
Lemma~\ref{lem:interpolation} uniformly in time.

\begin{lem}
\label{lem:flow-seed}
The solution $\psi(t)=\mathcal T_t^F\varphi$ satisfies, for every $0<\beta<\alpha$,
\beq
 \sup_{0\leq t\leq1}\|\psi(t)\|_{C^{1,\beta}}\leq C_\beta.
 \label{eq:flow-seed}
\eeq
\end{lem}

\begin{proof}
For the elliptic potentials in Proposition~\ref{prop:density-approximation}, we set
$\psi_\varepsilon(t)=\mathcal T_t^F\varphi_\varepsilon$. Their initial densities satisfy
$e^{-K}\leq f_\varepsilon\leq e^K$, and $\|\varphi_\varepsilon\|_\infty$ is uniformly bounded
by Lemma~\ref{lem:C0}. Thus Lemma~\ref{lem:flow-basic} applies with a constant
$B$ independent of $\varepsilon$. All these flows exist on $[0,1]$ and we use the same forcing
as $\psi$. By comparison,
$$
 \sup_{0\leq t\leq1}\|\psi_\varepsilon(t)-\psi(t)\|_\infty
 \leq\|\varphi_\varepsilon-\varphi\|_\infty
 \leq C\varepsilon^{1+\alpha}(1+|\log\varepsilon|).
$$
Combining \eqref{eq:trace-propagation} with \eqref{eq:density-approximation}, we obtain,
for every $\delta>0$,
\[
 \sup_{0\leq t\leq1}\|n+\Delta_g\psi_\varepsilon(t)\|_\infty
 \leq C\bigl(1+\|n+\Delta_g\varphi_\varepsilon\|_\infty\bigr)
 \leq C_\delta\varepsilon^{-(1-\alpha)(1+\delta)}.
\]
Lemma~\ref{lem:flow-basic} also bounds $\|\psi_\varepsilon(t)\|_\infty$ uniformly.
For $0<\gamma<1$ and finite $p>2n/(1-\gamma)$, the background elliptic estimate
\eqref{eq:background-elliptic} implies
\[
 \begin{aligned}
 \sup_{0\leq t\leq1}\|\psi_\varepsilon(t)\|_{C^{1,\gamma}}
 &\leq C_{p,\gamma}\sup_{0\leq t\leq1}
       \bigl(\|\Delta_g\psi_\varepsilon(t)\|_{L^p(d\mu_g)}
                       +\|\psi_\varepsilon(t)\|_\infty\bigr)\\
 &\leq C_{\gamma,\delta}\varepsilon^{-(1-\alpha)(1+\delta)}.
 \end{aligned}
\]
Here we use $|\Delta_g\psi_\varepsilon|\leq n+v_{\psi_\varepsilon}$ and
$\mu_g(M)=1$, so the constants are uniform in time and $\varepsilon$.
Fix $0<\beta<\alpha$ and choose $\rho,\delta,\gamma$ as in \eqref{parameters}, with
$a=1+\alpha-\rho$ and $b=(1-\alpha)(1+\delta)$. These choices satisfy
\[
 a>0,\qquad b>0,\qquad \beta<\gamma<1,\qquad
 a(\gamma-\beta)-b(1+\beta)=\frac{\alpha-\beta}{2}>0.
\]
After absorbing the logarithm as in \eqref{L infty app}, the approximation error is at most
$C_\rho\varepsilon^a$. The preceding spatial estimate bounds the norm by
$C_{\gamma,\delta}\varepsilon^{-b}$. Both bounds are uniform for $t\in[0,1]$. Applying
Lemma~\ref{lem:interpolation} with time as the additional parameter and $\varepsilon_0=1/4$,
we obtain \eqref{eq:flow-seed}, including both $t=0$ and $t=1$, with constants depending
only on $M,g,n,\alpha,K,\beta$.
\end{proof}

\subsection{Weighted energy estimates}
\label{sec:parabolic-energy}
In this subsection, we continue to study the smooth solution $\psi(t)=\mathcal T_t^F\varphi$
constructed in Section~\ref{sec:flow}. Under a uniform spatial $C^{1,\beta}$ bound on $[0,1]$, we
prove a trace estimate at positive times whose constant is independent of the initial trace.
For the $L^p$ estimate of the Laplacian, we draw on the weighted integral method of
Chen--Cheng \cite[Section~3]{ChenChengCscK1} for the constant scalar curvature equation.
In the parabolic setting, we combine a weighted energy inequality with spatial interpolation
to obtain an $L^p$ bound at positive times. The trace bound then follows by parabolic Moser
iteration. Together with Lemma~\ref{lem:interpolation}, this bound allows us to improve the
H\"older exponent on the whole time interval.

\begin{prop}
\label{prop:parabolic-trace}
Assume that $\psi(t)=\mathcal T_t^F\varphi$ satisfies
\begin{equation}
 \sup_{0\leq t\leq1}\|\psi(t)\|_{C^{1,\beta}}\leq B_\beta
 \quad\hbox{for some }0<\beta<1.
 \label{eq:flow-Holder-assumption}
\end{equation}
Then, with $d_\beta=(1-\beta)/(1+\beta)$, the trace satisfies, for every $\delta>0$,
\begin{equation}
 \|n+\Delta_g\psi(t)\|_\infty
 \leq C_{\beta,\delta}t^{-d_\beta-\delta},
 \qquad 0<t\leq1.
 \label{eq:parabolic-trace}
\end{equation}
The constant may also depend on $B_\beta$, but not on the initial trace or on higher derivatives of
$F$.
\end{prop}

To derive the weighted energy inequality, we combine the evolution of $w$ with a weight depending
on $U$. The time derivative of the volume measure contributes additional gradient terms, so it
must be included in the computation.

\begin{lem}
\label{lem:weighted-energy}
Let $\psi$ be a smooth solution of \eqref{eq:flow} satisfying \eqref{eq:flow-basic}. With $w$
as in Lemma~\ref{lem:trace-propagation}, set
\[
 h(U)=e^{U^2},\qquad
 Y_p(t)=\int_Mh(U)w^p\,d\mu_{\psi(t)}.
\]
There are positive constants $c,C$, independent of the real exponent $p\geq2$, such that, for $0<t\leq1$,
\begin{equation}
 Y_p'(t)+cp^2\int_Mw^{p-2}|\Na w|_\psi^2\,d\mu_\psi
 \leq Cp(1+t^{\alpha/2-1})Y_p(t).
 \label{eq:weighted-energy}
\end{equation}
The constants may also depend on the bound in \eqref{eq:flow-basic}.
\end{lem}

\begin{proof}
We divide the proof into three steps.

\paragraph{Step 1.}
Write $a=\log w$ and $b=\partial_t\widehat F$. By Lemma~\ref{lem:trace-propagation},
\eqref{eq:flow-identities} and \eqref{heat estimate V}, we have
\begin{equation}
\begin{split}
 \Dpar w&\leq L(t)w-w|\Na a|_\psi^2,
 \qquad L(t)\leq C(1+t^{\alpha/2-1}),\\
 \Dpar U&=-b,\qquad |b|\leq Ct^{\alpha/2-1}.
 \label{Dw L b}
\end{split}
\end{equation}
Since $w>0$, the chain rule applies to every real $p\geq2$. In particular,
\[
\begin{aligned}
 \Dpar(w^p)
 &=pw^{p-1}\Dpar w-p(p-1)w^{p-2}|\Na w|_\psi^2\\
 &\leq pL(t)w^p-p^2w^p|\Na a|_\psi^2.
\end{aligned}
\]
For $h=h(U)$, we also have
\[
 \Dpar h=-h'b-h''|\Na U|_\psi^2,
 \qquad \Na h=h'\Na U,\qquad \Na(w^p)=pw^p\Na a.
\]
Here primes on $h$ denote derivatives with respect to $U$. The product rule implies
\[
\begin{aligned}
 \Dpar(hw^p)
 &\leq(pLh-h'b)w^p-p^2hw^p|\Na a|_\psi^2
                      -h''w^p|\Na U|_\psi^2\\
 &\quad-2ph'w^p\langle\Na a,\Na U\rangle_\psi.
\end{aligned}
\]

\paragraph{Step 2.} We now differentiate $Y_p$, including the derivative of the volume form. Using
\eqref{eq:flow-identities} and Stokes' theorem, we obtain
\[
\begin{aligned}
 Y_p'
 &=\int_M\partial_t(hw^p)\,d\mu_\psi
                  +\int_Mhw^p\Delta_\psi U\,d\mu_\psi\\
 &=\int_M\Dpar(hw^p)\,d\mu_\psi
                  +\int_Mhw^p\Delta_\psi U\,d\mu_\psi.
\end{aligned}
\]
Integrating the last term by parts, we obtain
\[
 \int_Mhw^p\Delta_\psi U\,d\mu_\psi
 =-\int_Mw^p\bigl(ph\langle\Na a,\Na U\rangle_\psi
                         +h'|\Na U|_\psi^2\bigr)\,d\mu_\psi.
\]
Combining this identity with Step~1, we have
\[
 Y_p'\leq\int_M(pLh-h'b)w^p\,d\mu_\psi
                      -\int_Mw^p\mathcal Q\,d\mu_\psi,
\]
where
\[
 \mathcal Q=p^2h|\Na a|_\psi^2
    +p(2h'+h)\langle\Na a,\Na U\rangle_\psi
    +(h''+h')|\Na U|_\psi^2.
\]

\paragraph{Step 3.}
For $h(U)=e^{U^2}$, the derivative ratios are $h'/h=2U$ and $h''/h=2+4U^2$.
We calculate
\[
\begin{aligned}
 \frac{\mathcal Q}{h}
 &=p^2|\Na a|_\psi^2
    +p(4U+1)\langle\Na a,\Na U\rangle_\psi
    +(4U^2+2U+2)|\Na U|_\psi^2\\
 &=\left|p\Na a+\left(2U+\frac12\right)\Na U\right|_\psi^2
                                 +\frac74|\Na U|_\psi^2.
\end{aligned}
\]
Choose $B_1$ such that $|U|\leq B_1$, as provided by \eqref{eq:flow-basic}. Then we have
\[
\begin{aligned}
 |p\Na a|_\psi^2
 &\leq2\left|p\Na a+\left(2U+\frac12\right)\Na U\right|_\psi^2
       +2\left(2B_1+\frac12\right)^2|\Na U|_\psi^2\\
 &\leq\left[2+\frac87\left(2B_1+\frac12\right)^2\right]
                \frac{\mathcal Q}{h}.
\end{aligned}
\]
Thus $\mathcal Q\geq cp^2h|\Na a|_\psi^2$, with $c>0$ independent of $p$.
Moreover, $1\leq h\leq e^{B_1^2}$ and $|h'|\leq2B_1h$. The bounds in \eqref{Dw L b} imply
\[
 pLh-h'b\leq Cp(1+t^{\alpha/2-1})h,\qquad p\geq2.
\]
Now $w^ph|\Na a|_\psi^2=hw^{p-2}|\Na w|_\psi^2\geq w^{p-2}|\Na w|_\psi^2$. Substituting these
bounds into the inequality for $Y_p'$, we obtain \eqref{eq:weighted-energy}.
\end{proof}

The next estimate controls the spatial energy at a fixed time. The proof uses heat regularization
with respect to the background metric and does not require upper or lower bounds for the
determinant of the varying metric.

\begin{lem}
\label{lem:spatial-interpolation}
Fix $0<\beta<1$ and a finite real number $p\geq2$. For a smooth K\"ahler potential $\varphi$ with
$\|\varphi\|_{C^{1,\beta}}\leq B_\beta$, set
\[
 v=n+\Delta_g\varphi,\qquad X=\int_Mv^p\,d\mu_g,\qquad
 \mathcal E_p=\int_Mv^{p-3}|\Na v|_g^2\,d\mu_g,
 \qquad d_\beta=\frac{1-\beta}{1+\beta}.
\]
Then the following inequality holds:
\begin{equation}
 \mathcal E_p\geq cX^{\,1+1/(pd_\beta)}-CX.
 \label{eq:spatial-interpolation}
\end{equation}
The constants $c,C>0$ may depend on $M,g,n,p,\beta,B_\beta$.
\end{lem}

\begin{proof}
We first estimate the moment of order $p+1$ at a fixed regularization scale, and then choose this
scale in terms of $X$.

\paragraph{Step 1.}
For $0<r\leq1$, set $\varphi_r=P_{r^2}\varphi$. By the background heat estimates
\eqref{eq:heat-gradient}, we have
\begin{align}\label{nabla vphi vphi r}
 \|\Na(\varphi-\varphi_r)\|_\infty\leq Cr^\beta,
 \qquad \|\Delta_g\varphi_r\|_\infty\leq Cr^{\beta-1}.
\end{align}
Since $v>0$ by the K\"ahler condition, all the real powers of $v$ below are defined, including
$v^{p-3}$ when $2\leq p<3$. Set $Z=\int_Mv^{p+1}\,d\mu_g$ and write
$v=n+\Delta_g\varphi_r+\Delta_g(\varphi-\varphi_r)$. Integrating by parts and using
\eqref{nabla vphi vphi r}, we obtain
\[
\begin{aligned}
 Z
 &=nX+\int_Mv^p\Delta_g\varphi_r\,d\mu_g
                +\int_Mv^p\Delta_g(\varphi-\varphi_r)\,d\mu_g\\
 &=nX+\int_Mv^p\Delta_g\varphi_r\,d\mu_g
     -p\int_Mv^{p-1}
       \langle\Na v,\Na(\varphi-\varphi_r)\rangle_g\,d\mu_g\\
 &\leq(n+Cr^{\beta-1})X
             +Cpr^\beta\int_Mv^{p-1}|\Na v|_g\,d\mu_g.
\end{aligned}
\]
Since $r\leq1$ and $\beta<1$, the first term is at most $Cr^{\beta-1}X$. Factoring
$v^{p-1}|\Na v|_g=(v^{(p-3)/2}|\Na v|_g)v^{(p+1)/2}$ and applying the Cauchy--Schwarz inequality to
the second term, we obtain
\[
 \int_Mv^{p-1}|\Na v|_g\,d\mu_g
 \leq\mathcal E_p^{1/2}Z^{1/2}.
\]
Consequently,
\[
 Z\leq Cr^{\beta-1}X+Cpr^\beta\mathcal E_p^{1/2}Z^{1/2}.
\]
To absorb the term involving $Z^{1/2}$, use Young's inequality in the form
$ Cpr^\beta\mathcal E_p^{1/2}Z^{1/2} \leq\frac12Z+C_p r^{2\beta}\mathcal E_p. $ Thus
\begin{align}\label{bound Z}
 Z\leq C_1r^{\beta-1}X+C_2r^{2\beta}\mathcal E_p.
\end{align}

\paragraph{Step 2.}
Since the measure $d\mu_g$ has mass one, H\"older's inequality implies
\[
 X=\int_M(v^{p+1})^{p/(p+1)}\,d\mu_g
       \leq Z^{p/(p+1)},
 \qquad Z\geq X^{1+1/p}.
\]
It follows from \eqref{bound Z} that
\[
 X^{1+1/p}\leq C_1r^{\beta-1}X+C_2r^{2\beta}\mathcal E_p.
\]
Since $\beta<1$, a fixed $L\geq1$ can be chosen so large that $C_1L^{\beta-1}\leq1/2$. When
$X\geq X_0:=L^{p(1-\beta)}$, we set
\[
 r=LX^{-1/(p(1-\beta))}\leq1.
\]
For this choice of $r$, we have
\[
 r^{\beta-1}X=L^{\beta-1}X^{1+1/p},\qquad
 r^{2\beta}=L^{2\beta}X^{-2\beta/(p(1-\beta))}.
\]
Absorbing the first term into the left-hand side, we obtain
\[
 \frac12X^{1+1/p}\leq
 C_2L^{2\beta}X^{-2\beta/(p(1-\beta))}\mathcal E_p.
\]
It follows that
\[
\begin{aligned}
 \mathcal E_p
 &\geq cX^{1+1/p+2\beta/(p(1-\beta))}
 =cX^{1+(1+\beta)/(p(1-\beta))}
  =cX^{1+1/(pd_\beta)}.
\end{aligned}
\]
For $0<X<X_0$, choose $C\geq cX_0^{1/(pd_\beta)}$. Then $cX^{1+1/(pd_\beta)}-CX\leq0$,
so \eqref{eq:spatial-interpolation} follows from the nonnegativity of $\mathcal E_p$.
\end{proof}

Combining the spatial lower bound with the weighted energy inequality, we obtain a differential
inequality with a superlinear dissipative term. By comparing the trace moment with a singular barrier, we obtain a bound at positive times
with a constant independent of the initial trace moment.

\begin{lem}
\label{lem:parabolic-Lp}
Under \eqref{eq:flow-Holder-assumption}, the following estimate holds for every finite $p\geq2$:
\begin{equation}
 \int_M(n+\Delta_g\psi(t))^p\,d\mu_g
 \leq C_{p,\beta}t^{-pd_\beta},\qquad
 d_\beta=\frac{1-\beta}{1+\beta},\quad 0<t\leq1.
 \label{eq:parabolic-Lp}
\end{equation}
The constant may also depend on $B_\beta$, but not on the initial trace.
\end{lem}

\begin{proof}
We fix $p\geq2$ and divide the proof into three steps. Let $X(t)$ and $\mathcal E_p(t)$ be the
quantities in Lemma~\ref{lem:spatial-interpolation} for the potential $\psi(t)$.

\paragraph{Step 1.}
It follows from $w=e^{-A\psi}v$ and $d\mu_\psi=e^Hd\mu_g$ that
\[
 Y_p(t)=\int_Me^{U^2-pA\psi+H}v^p\,d\mu_g.
\]
If $|\psi|\leq B_0$, $|U|\leq B_1$, and $|H|\leq B_2$, then we have
\begin{align}\label{lem:parabolic-Lp X Yp}
 e^{-pAB_0-B_2}X(t)
 \leq Y_p(t)\leq e^{B_1^2+pAB_0+B_2}X(t).
\end{align}
Thus the two quantities are comparable, with $C_p^{-1}X\leq Y_p\leq C_pX$.
Differentiating $v=e^{A\psi}w$, we have $ \Na v=e^{A\psi}(\Na w+Aw\Na\psi). $
Combining this identity with $|\Na\psi|_g\leq B_\beta$, we obtain
\[
\begin{aligned}
 v^{p-3}|\Na v|_g^2
 &\leq 2e^{(p-1)A\psi}w^{p-3}|\Na w|_g^2
       +2A^2e^{(p-1)A\psi}w^{p-1}|\Na\psi|_g^2\\
 &\leq C_pw^{p-3}|\Na w|_g^2+C_pw^{p-1}.
\end{aligned}
\]
By \eqref{eq:flow-density} and the arithmetic--geometric mean inequality, $v\geq ne^{H/n}\geq c>0$.
Hence $w\geq c>0$ and $w^{p-1}\leq c^{-1}w^p$. The algebraic comparison
\eqref{eq:algebraic-comparisons} implies
\[
 |\Na w|_g^2\leq v|\Na w|_\psi^2
                    \leq Cw|\Na w|_\psi^2.
\]
After integration and comparison of the measures, we obtain
\[
 \mathcal E_p(t)
 \leq C_p\int_Mw^{p-2}|\Na w|_\psi^2\,d\mu_\psi+C_pY_p(t).
\]

\paragraph{Step 2.}
By \eqref{eq:flow-Holder-assumption}, Lemma~\ref{lem:spatial-interpolation} applies to $\psi(t)$.
Using \eqref{lem:parabolic-Lp X Yp} to express the lower bound in terms of $Y_p(t)$, we obtain
\[
 \mathcal E_p(t)
 \geq cX(t)^{1+1/(pd_\beta)}-CX(t)
 \geq c'Y_p(t)^{1+1/(pd_\beta)}-C'Y_p(t).
\]
The estimate for $\mathcal E_p$ in Step~1 now implies
\[
 \int_Mw^{p-2}|\Na w|_\psi^2\,d\mu_\psi
 \geq cY_p^{1+1/(pd_\beta)}-CY_p.
\]
All constants at this stage may depend on $p$ and $\beta$. Substituting the lower bound for
the gradient integral into \eqref{eq:weighted-energy}, we obtain
\[
 Y_p'+cY_p^{1+1/(pd_\beta)}
   \leq C(1+t^{\alpha/2-1})Y_p
   \leq \frac CtY_p,
 \qquad 0<t\leq1.
\]
The last inequality uses $1\leq t^{-1}$ and $t^{\alpha/2-1}\leq t^{-1}$ on $(0,1]$.

\paragraph{Step 3.} We set $m=pd_\beta>0$ and choose $L>0$ so that $cL^{1/m}>C+m$. For $B(t)=Lt^{-m}$, we have
\[
 B'(t)=-\frac mtB(t),\qquad
 B(t)^{1+1/m}=\frac{L^{1/m}}tB(t).
\]
Consequently, we have
\[
 B'+cB^{1+1/m}-\frac CtB
       =\frac{cL^{1/m}-m-C}{t}B>0.
\]
For each fixed smooth solution, the function $Y_p$ is continuous at zero, with $Y_p(0)<\infty$,
whereas the barrier $B(t)$ diverges as $t\downarrow0$. Thus $Y_p(t)<B(t)$ for all sufficiently small
positive times. If a first contact occurred at $t_0>0$, then $Y_p(t_0)=B(t_0)$ and
$(Y_p-B)'(t_0)\geq0$. Subtracting the differential inequalities at the point of contact, we obtain
\[
 (Y_p-B)'(t_0)
 \leq\frac C{t_0}B(t_0)-cB(t_0)^{1+1/m}-B'(t_0)<0,
\]
a contradiction. Hence $Y_p(t)\leq Lt^{-m}$ for every $0<t\leq1$.

The time at which the comparison begins can depend on the individual solution, but $L$ is determined solely
by the constants in the differential inequality and is independent of $Y_p(0)$ or the initial trace. The
comparison of $X$ and $Y_p$ in Step~1 now proves \eqref{eq:parabolic-Lp}.
\end{proof}

We now combine the weighted energy inequality and the fixed-exponent estimate to prove
Proposition~\ref{prop:parabolic-trace}.

\begin{proof}[Proof of Proposition~\ref{prop:parabolic-trace}]
We fix $0<T\leq1$ and divide the proof into four steps. All space-time norms below are taken with
respect to $d\mu_{\psi(t)}\,dt$.

\paragraph{Step 1.}
The density bound \eqref{eq:flow-density} is uniform, and $[\omega_{\psi(t)}]=[\omega_g]$.
Thus the Sobolev inequality \eqref{eq:Sobolev-estimate} applies for any fixed $1<\nu<n/(n-1)$.
In particular,
\[
 \left(\int_M|b|^{2\nu}\,d\mu_{\psi(t)}\right)^{1/\nu}
 \leq C_\nu\left(
       \int_M|\Na b|_\psi^2\,d\mu_{\psi(t)}
       +\int_M|b|^2\,d\mu_{\psi(t)}\right).
\]
The constant is uniform in $t$ and independent of the initial trace. We define
\[
 \kappa=2-\frac1\nu>1,\qquad
 \sigma=\sum_{j=0}^{\infty}\kappa^{-j}
        =\frac{\kappa}{\kappa-1}.
\]
On $[T/2,T]$, the coefficient in \eqref{eq:weighted-energy} satisfies
\[
 1+t^{\alpha/2-1}\leq\frac CT.
\]

We take $T/2\leq a<b<T$ and a nondecreasing Lipschitz cutoff $\chi:[a,T]\to[0,1]$ such that
\[
 \chi(a)=0,\qquad \chi=1\text{ on }[b,T],\qquad
 0\leq\chi'\leq\frac{2}{b-a}\quad\text{almost everywhere}.
\]
Since $\chi$ is Lipschitz, $\chi Y_p$ is absolutely continuous and the product rule holds almost
everywhere. Multiplying \eqref{eq:weighted-energy} by $\chi$ and using this rule, we obtain
\[
 (\chi Y_p)'
 +cp^2\chi\int_Mw^{p-2}|\Na w|_\psi^2\,d\mu_\psi
 \leq\left(\chi'+\frac{Cp}{T}\chi\right)Y_p.
\]
The cutoff vanishes at $a$ and equals one at every $s\in[b,T]$. Integration from $a$ to $s$ therefore
implies
\[
\begin{aligned}
 Y_p(s)
 &+cp^2\int_b^s\int_Mw^{p-2}|\Na w|_\psi^2\,d\mu_\psi\,dt
 \leq\left(\frac{2}{b-a}+\frac{Cp}{T}\right)
                       \int_a^TY_p(t)\,dt.
\end{aligned}
\]
To bound $\sup_{[b,T]}Y_p$, we drop the gradient term and take the supremum over $s$. To bound the
gradient integral on $[b,T]$, we take $s=T$ and drop $Y_p(T)$. Adding the resulting inequalities,
we use $1\leq h\leq e^{B_1^2}$ and
\[
 |\Na(w^{p/2})|_\psi^2
       =\frac{p^2}{4}w^{p-2}|\Na w|_\psi^2,
\]
to obtain, after enlarging the coefficient to $Cp^2$,
\begin{equation}
\begin{aligned}
 &\sup_{b\leq t\leq T}\int_Mw^p\,d\mu_\psi
       +\int_b^T\int_M|\Na(w^{p/2})|_\psi^2\,d\mu_\psi\,dt\\
 &\qquad\leq Cp^2\left(\frac1{b-a}+\frac1T\right)
                 \int_a^T\int_Mw^p\,d\mu_\psi\,dt.
\end{aligned}
 \label{eq:parabolic-cutoff}
\end{equation}
The constant $C$ in this inequality is independent of $p\geq2$.

\paragraph{Step 2.}
On each time slice, we write $w^{p\kappa}=w^{p(\kappa-1)}w^p$. Applying H\"older's inequality with
conjugate exponents $\nu/(\nu-1)$ and $\nu$, we obtain
\[
\begin{aligned}
 \int_Mw^{p\kappa}\,d\mu_\psi
 &\leq
 \left(\int_Mw^{p(\kappa-1)\nu/(\nu-1)}\,d\mu_\psi
       \right)^{(\nu-1)/\nu}
 \left(\int_Mw^{p\nu}\,d\mu_\psi\right)^{1/\nu}\\
 &=\left(\int_Mw^p\,d\mu_\psi\right)^{\kappa-1}
 \left(\int_Mw^{p\nu}\,d\mu_\psi\right)^{1/\nu},
\end{aligned}
\]
since $\kappa-1=(\nu-1)/\nu$. The spatial Sobolev inequality for $w^{p/2}$ then bounds the second
factor. Integrating in time, we obtain
\[
\begin{aligned}
 &\int_b^T\int_Mw^{p\kappa}\,d\mu_\psi\,dt\\
 &\quad\leq C_\nu
 \left(\sup_{b\leq t\leq T}\int_Mw^p\,d\mu_\psi\right)^{\kappa-1}\\
 &\qquad\quad\times\left(
 \int_b^T\int_M|\Na(w^{p/2})|_\psi^2\,d\mu_\psi\,dt
            +\int_b^T\int_Mw^p\,d\mu_\psi\,dt\right).
\end{aligned}
\]
The cutoff estimate \eqref{eq:parabolic-cutoff} controls the supremum and gradient integral. It also
absorbs the last zeroth-order integral after increasing $C$, since $p\geq2$, $T\leq1$, and
$Cp^2((b-a)^{-1}+T^{-1})\geq1$. It follows that
\[
 \int_b^T\int_Mw^{p\kappa}\,d\mu_\psi\,dt
 \leq
 \left[Cp^2\left(\frac1{b-a}+\frac1T\right)
       \int_a^T\int_Mw^p\,d\mu_\psi\,dt\right]^\kappa.
\]
Taking the $1/(p\kappa)$ power, we obtain
\begin{equation}
 \|w\|_{L^{p\kappa}(M\times[b,T])}
 \leq
 \left[Cp^2\left(\frac1{b-a}+\frac1T\right)\right]^{1/p}
 \|w\|_{L^p(M\times[a,T])}.
 \label{eq:parabolic-Moser-step}
\end{equation}
The constant in \eqref{eq:parabolic-Moser-step} is independent of $p$, $a$, $b$, and $T$, once $\nu$
is fixed.

\paragraph{Step 3.}
We choose a finite $p_0\geq2$ and set
\[
 p_j=p_0\kappa^j,\qquad
 a_j=T\left(\frac34-2^{-j-2}\right),\qquad j\geq0.
\]
Then $a_0=T/2$, $a_j\uparrow3T/4$, and
\[
 a_{j+1}-a_j=T2^{-j-3}.
\]
Applying \eqref{eq:parabolic-Moser-step} with $p=p_j$, $a=a_j$, and $b=a_{j+1}$, we obtain
\[
 \|w\|_{L^{p_{j+1}}(M\times[a_{j+1},T])}
 \leq\left(\frac{C p_j^2 2^{j+4}}{T}\right)^{1/p_j}
        \|w\|_{L^{p_j}(M\times[a_j,T])}.
\]
The iteration produces a product of the displayed factors. Their logarithms are summable, because
\[
\begin{gathered}
 \sum_{j=0}^{\infty}\frac1{p_j}=\frac{\sigma}{p_0},\qquad
 \sum_{j=0}^{\infty}\frac{j}{p_j}
       =\frac{\kappa}{p_0(\kappa-1)^2}<\infty,\\
 \log p_j=\log p_0+j\log\kappa.
\end{gathered}
\]
In particular, the products satisfy
\[
 \prod_{j=0}^{\infty}
       (C p_j^2 2^{j+4})^{1/p_j}\leq C_{p_0,\nu},
 \qquad
 \prod_{j=0}^{\infty}T^{-1/p_j}=T^{-\sigma/p_0}.
\]

Before passing to the limit as $j\to\infty$, restrict the norms on the left to the fixed cylinder
$M\times[3T/4,T]$, which lies in every cylinder of the iteration. Its space-time measure is $T/4$
because $\int_Md\mu_{\psi(t)}=1$. On this fixed cylinder the $L^{p_j}$ norms converge to the
essential supremum. Since the density is positive and $w$ is continuous up to $t=T$, this is the
supremum. Hence
\[
 \sup_{M\times[3T/4,T]}w
 \leq C_{p_0,\nu}T^{-\sigma/p_0}
       \left(\int_{T/2}^T\int_Mw^{p_0}\,d\mu_\psi\,dt\right)^{1/p_0}.
\]

\paragraph{Step 4.}
Lemma~\ref{lem:parabolic-Lp} provides the bound at the fixed exponent $p_0$. Using the comparability
of $w$ and $v$ and of the measures, we obtain
\[
\begin{aligned}
 \int_{T/2}^T\int_Mw^{p_0}\,d\mu_\psi\,dt
 &\leq C_{p_0,\beta}\int_{T/2}^Tt^{-p_0d_\beta}\,dt
 \leq C_{p_0,\beta}T^{1-p_0d_\beta}.
\end{aligned}
\]
Here $t$ is comparable to $T$ throughout the interval of integration. The estimate in Step~3
therefore implies
\[
\begin{aligned}
 \|w(T)\|_\infty
 &\leq C_{p_0,\nu,\beta}
          T^{-\sigma/p_0}T^{(1-p_0d_\beta)/p_0}
 =C_{p_0,\nu,\beta}T^{-d_\beta-(\sigma-1)/p_0}.
\end{aligned}
\]
For a prescribed $\delta>0$, we choose
\[
 p_0>\max\left\{2,\frac{\sigma-1}{\delta}\right\}.
\]
Since $T\leq1$, this proves $\|w(T)\|_\infty\leq C_{\beta,\delta}T^{-d_\beta-\delta}$. Since
$v=e^{A\psi}w\leq Cw$, the estimate \eqref{eq:parabolic-trace} follows. All constants used above are
independent of the initial trace and of higher derivatives of $F$.
\end{proof}

To improve the H\"older exponent on the entire time interval, use the trace estimate at a small
positive time to control new initial data. The new flows start at time zero from these data and
use the original forcing. This allows us to compare them with the original flow.

\begin{prop}
\label{prop:bootstrap}
If \eqref{eq:flow-Holder-assumption} holds for some $0<\beta<1$, then, for every
\begin{equation}
 0<\beta'<\mathcal B_\alpha(\beta),\qquad
 \mathcal B_\alpha(\beta)
     =\frac{\alpha+(4+\alpha)\beta}{4+\alpha+\alpha\beta},
 \label{eq:bootstrap-map}
\end{equation}
we have
\[
 \sup_{0\leq t\leq1}\|\psi(t)\|_{C^{1,\beta'}}
       \leq C_{\beta,\beta'}.
\]
The constant may also depend on $B_\beta$, but not on the initial trace or on higher derivatives of
$F$.
\end{prop}

\begin{proof}
We divide the proof into three steps.

\paragraph{Step 1.}
We set
\[
 a=1+\frac\alpha2,\qquad d_\beta=\frac{1-\beta}{1+\beta},\qquad
 z_\tau=\psi(\tau),\qquad 0<\tau\leq\frac14.
\]
The original flow starts at the elliptic solution $\varphi$. By \eqref{eq:flow-approximation},
\[
 \|z_\tau-\varphi\|_\infty\leq C\tau^a.
\]
For every fixed $d>d_\beta$, we apply Proposition~\ref{prop:parabolic-trace} with
$\delta=d-d_\beta$ to obtain
\[
 \|n+\Delta_gz_\tau\|_\infty\leq C_d\tau^{-d}.
\]
The basic estimates also imply
\[
 \|z_\tau\|_\infty+\|\log f_{z_\tau}\|_\infty
 =\|\psi(\tau)\|_\infty+\|H(\tau)\|_\infty\leq C,
\]
with a constant independent of $\tau$.

\paragraph{Step 2.}
For each $\tau$, we define
\[
 \zeta_\tau(t)=\mathcal T_t^Fz_\tau,
 \qquad
 \begin{cases}
 \displaystyle
 \partial_t\zeta_\tau
     =\log\frac{\omega_{\zeta_\tau(t)}^n}{\omega_g^n}-P_tF,\\[2mm]
 \zeta_\tau(0)=z_\tau.
 \end{cases}
\]
By Lemma~\ref{lem:flow-basic}, the solution is smooth on $[0,1]$, and the potential and its
logarithmic density are uniformly bounded. Each new flow starts at time zero with forcing $P_tF$.
Since $\zeta_\tau$ and $\psi$ have the same forcing, the comparison estimate
\eqref{eq:flow-contraction} implies
\[
 \sup_{0\leq t\leq1}\|\zeta_\tau(t)-\psi(t)\|_\infty
 \leq\|z_\tau-\varphi\|_\infty\leq C\tau^a.
\]
By Lemma~\ref{lem:trace-propagation} and Step~1, we also have
$$
 \sup_{0\leq t\leq1}\|n+\Delta_g\zeta_\tau(t)\|_\infty
 \leq C\bigl(1+\|n+\Delta_gz_\tau\|_\infty\bigr)\leq C_d\tau^{-d}.
$$
The constants in these applications of the flow estimates are uniform, since
$\|z_\tau\|_\infty+\|\log f_{z_\tau}\|_\infty\leq C$.

Fix $0<\gamma<1$ and choose a finite $p>2n/(1-\gamma)$. Applying the background elliptic estimate
and Sobolev embedding \eqref{eq:background-elliptic}, as in Gilbarg--Trudinger
\cite[Chapters~7 and 9]{GilTru}, we obtain
$$
 \|\zeta_\tau(t)\|_{C^{1,\gamma}}
 \leq C_\gamma\bigl(\|\Delta_g\zeta_\tau(t)\|_\infty
                              +\|\zeta_\tau(t)\|_\infty\bigr)
 \leq C_{\gamma,d}\tau^{-d}.
$$
Both the approximation estimate and the $C^{1,\gamma}$ estimate are uniform for $0\leq t\leq1$.

\paragraph{Step 3.}
We now choose $d$ and $\gamma$ so that Lemma~\ref{lem:interpolation} applies. The limiting threshold
is
\[
\begin{aligned}
 \frac{a-d_\beta}{a+d_\beta}
 &=\frac{(2+\alpha)(1+\beta)-2(1-\beta)}
          {(2+\alpha)(1+\beta)+2(1-\beta)}\\
 &=\frac{\alpha+(4+\alpha)\beta}{4+\alpha+\alpha\beta}
  =\mathcal B_\alpha(\beta).
\end{aligned}
\]
Fix $0<\beta'<\mathcal B_\alpha(\beta)$. The upper bound is equivalent to
$d_\beta<a(1-\beta')/(1+\beta')$. Thus we may choose
\[
 d_\beta<d<\frac{a(1-\beta')}{1+\beta'},\qquad
 \beta'+\frac{d(1+\beta')}{a}<\gamma<1.
\]
For these choices, $0<\beta'<\gamma<1$, and
\[
 a(\gamma-\beta')>d(1+\beta'),\qquad
 \beta'<\frac{a\gamma-d}{a+d}.
\]
The last inequality is the strict exponent condition in Lemma~\ref{lem:interpolation}, and the
estimates in Step~2 are uniform for $0\leq t\leq1$. Applying the lemma with $\varepsilon=\tau$,
$b=d$, $\varphi_\varepsilon=\zeta_\tau(t)$ and $\varphi=\psi(t)$, we obtain
\[
 \sup_{0\leq t\leq1}\|\psi(t)\|_{C^{1,\beta'}}\leq C_{\beta,\beta'}.
\]
The parameters $d$ and $\gamma$ can be fixed in terms of $\alpha,\beta,\beta'$, so the constant has
the asserted dependence. This proves the proposition, including both time endpoints.

Finally, the bootstrap map satisfies
\[
\begin{aligned}
 \mathcal B_\alpha(\beta)-\beta
 &=\frac{\alpha(1-\beta^2)}{4+\alpha+\alpha\beta}>0,\\
 1-\mathcal B_\alpha(\beta)
 &=\frac{4(1-\beta)}{4+\alpha+\alpha\beta}>0.
\end{aligned}
\]
Thus $\beta<\mathcal B_\alpha(\beta)<1$. We may choose $\beta'$ strictly between $\beta$ and
$\mathcal B_\alpha(\beta)$, and the improved bound supplies the hypothesis for the next application
of Proposition~\ref{prop:parabolic-trace}.
\end{proof}

\section{Proof of Theorem~\ref{thm:main}}
\label{sec:conclusion}
We begin with the smooth case of Theorem~\ref{thm:main}. After finitely many applications of
Proposition~\ref{prop:bootstrap}, the exponent lies in the strict range in \eqref{eq:LLZ}, so the
local theorem of Li--Li--Zhang applies. For general $C^\alpha$ data, we approximate the density
and pass the estimate to the limit with the original H\"older exponent.

\begin{proof}[Proof of Theorem~\ref{thm:main} for smooth data]
We divide the proof into four steps.

\paragraph{Step 1.}
Assume that $F$ is smooth and satisfies \eqref{eq:data}. By Yau \cite{Yau}, there is a smooth
solution $\varphi$ of \eqref{eq:MA} with zero $d\mu_g$ mean. We set
\[
 \psi(t)=\mathcal T_t^F\varphi,\qquad 0\leq t\leq1.
\]
By Lemma~\ref{lem:flow-basic}, the flow exists on $[0,1]$. Lemma~\ref{lem:flow-seed} implies
\[
 \sup_{0\leq t\leq1}\|\psi(t)\|_{C^{1,\beta_0}}\leq C_0,
 \qquad \beta_0=\frac\alpha2.
\]
Here $C_0$ depends only on $M,g,n,\alpha,K$.
Proposition~\ref{prop:bootstrap} improves a uniform
$C^{1,\beta}$ bound on $[0,1]$ to a uniform $C^{1,\beta'}$ bound on the same interval whenever
$\beta<\beta'<\mathcal B_\alpha(\beta)$, with the map defined in \eqref{eq:bootstrap-map}. We recall
that
\begin{equation}
 \mathcal B_\alpha(\beta)-\beta
 =\frac{\alpha(1-\beta^2)}{4+\alpha+\alpha\beta}>0,
 \qquad 0<\beta<1.
 \label{eq:strict-improvement}
\end{equation}
Note that we have
\[
 \frac{1-\mathcal B_\alpha(\beta)}{1+\mathcal B_\alpha(\beta)}
 =\frac1{1+\alpha/2}\frac{1-\beta}{1+\beta}.
\]
Set
\[
\begin{gathered}
 k=1+\frac\alpha4,\qquad
 d_j=\frac{2-\alpha}{2+\alpha}\,k^{-j},\\
 \beta_j=\frac{1-d_j}{1+d_j},\qquad j=0,1,2,\ldots.
\end{gathered}
\]
The initial exponent is $\beta_0=\alpha/2$. Since $1<k<1+\alpha/2$, it follows that
\[
 \frac{d_j}{1+\alpha/2}<d_{j+1}=\frac{d_j}{k}<d_j.
\]
The map $d\mapsto(1-d)/(1+d)$ is strictly decreasing for $d\geq0$. Thus, we have
\[
 \beta_j<\beta_{j+1}<\mathcal B_\alpha(\beta_j)<1,
\]
so every successive exponent is admissible for Proposition~\ref{prop:bootstrap}.

\paragraph{Step 2.}
It remains to choose the number of iterations needed to reach the local threshold in
\eqref{eq:LLZ}. Write
\[
 \beta_*(n,\alpha)=1-\theta,\qquad
 \theta=\frac{\alpha}{n(2+\alpha)-1}.
\]
For $d\geq0$, we have
\[
\begin{aligned}
 \frac{1-d}{1+d}>1-\theta
 &\quad\Longleftrightarrow\quad d<\frac{\theta}{2-\theta},\\
 \frac{\theta}{2-\theta}
 &=\frac{\alpha}{2n(2+\alpha)-2-\alpha}
  =\frac{\alpha}{(2n-1)(2+\alpha)}.
\end{aligned}
\]
We choose
\[
 J=1+\left\lfloor
 \frac{\log\bigl((2-\alpha)(2n-1)/\alpha\bigr)}
      {\log(1+\alpha/4)}
 \right\rfloor.
\]
The quotient is positive for $n\geq2$ and $0<\alpha<1$. By the definition of $J$,
\[
 k^J>\frac{(2-\alpha)(2n-1)}\alpha,
 \qquad
 d_J<\frac{\alpha}{(2n-1)(2+\alpha)}.
\]
In particular,
\[
 \beta_J-\beta_*(n,\alpha)
 =\frac{\theta-(2-\theta)d_J}{1+d_J}>0.
\]
Both $J$ and this positive gap are fixed in terms of $n,\alpha$.

Starting with the bound at $\beta_0$ and applying Proposition~\ref{prop:bootstrap} to the pairs
$\beta_j,\beta_{j+1}$, we obtain inductively
\[
 \sup_{0\leq t\leq1}\|\psi(t)\|_{C^{1,\beta_j}}\leq C_j,
 \qquad 0\leq j\leq J.
\]
At each step, the new constant depends on the preceding bound and the fixed exponents. Since the
exponents depend only on $\alpha$ and $J$ depends only on $n,\alpha$, the constant $C_J$ depends
only on $M,g,n,\alpha,K$. Each estimate holds on the whole interval $[0,1]$. Taking $t=0$, we
therefore obtain
\[
 \|\varphi\|_{C^{1,\beta_J}(M,g)}\leq C_J.
\]
This uses finitely many applications of the quantitative estimates in Lemma~\ref{lem:flow-seed} and
Proposition~\ref{prop:bootstrap}, whose constants are independent of higher derivatives of $F$.

\paragraph{Step 3.}
To write the equation in local coordinates, we choose finitely many holomorphic charts with images
$\Omega_\ell\subset\mathbb C^n$ and smaller balls $\Omega'_\ell\Subset\Omega_\ell$ whose inverse
images cover $M$. Each $\Omega_\ell$ is a ball with compact closure in the image of a larger
chart that admits a smooth local potential of $\omega_g$. Fix these charts, so that the coordinate
changes and the background quantities are fixed throughout the argument.
On one such chart, let $\rho$ be a fixed smooth real function with $\rho_{i\bar j}=g_{i\bar j}$ and
set $\Phi=\rho+\varphi$. In these coordinates, we have
\[
 (\Phi_{i\bar j})=(g_{i\bar j}+\varphi_{i\bar j})>0,
 \qquad \det(\Phi_{i\bar j})=a,
 \qquad a=e^F\det(g_{i\bar j}).
\]
Thus $\Phi$ is plurisubharmonic and solves the local complex Monge--Amp\`ere equation.
On the finitely many larger balls, the background determinants have a common positive lower
bound. Using $|F|\leq K$, we obtain
\[
 a\geq\lambda_0,
 \qquad
 \lambda_0=e^{-K}\min_\ell\inf_{\Omega_\ell}\det(g_{i\bar j})>0.
\]
The exponential is Lipschitz on the range of $F$. By the mean value theorem and the
equivalence of the background and coordinate norms,
\[
 \|e^F\|_{C^\alpha(\Omega_\ell)}\leq Ce^K(1+K).
\]
Using the product inequality for H\"older norms and the fixed smoothness of the background, we
obtain
\[
 \|a\|_{C^\alpha(\Omega_\ell)}\leq Ce^K(1+K).
\]
The bound from Step~2 also implies
\[
 \|\Phi\|_{C^{1,\beta_J}(\Omega_\ell)}
 \leq C\bigl(1+\|\varphi\|_{C^{1,\beta_J}(M,g)}\bigr)
 \leq C(1+C_J).
\]
These bounds are uniform over the finite atlas and depend only on the stated data.

\paragraph{Step 4.}
Since $\beta_J>\beta_*(n,\alpha)$, we can apply the local theorem of Li--Li--Zhang
\cite[Theorem~1.2]{LLZ} to obtain
\[
 \|\Phi\|_{C^{2,\alpha}(\Omega'_\ell)}\leq C.
\]
The constant depends on $n,\alpha,\beta_J$, the lower bound $\lambda_0$, the two H\"older norms
estimated in Step~3, and $\operatorname{dist}(\Omega'_\ell,\partial\Omega_\ell)$. These quantities
are controlled in terms of $M,g,n,\alpha,K$. In particular, the estimate bounds all second
real coordinate derivatives.

After subtracting the fixed local potential $\rho$, we have the same estimate for $\varphi$
on each smaller ball. The local norms control the background norm by a finite partition of
unity and the fixed smooth coordinate changes. Thus, we have
\[
 \|\varphi\|_{C^{2,\alpha}(M,g)}\leq C(M,g,n,\alpha,K).
\]
This proves \eqref{eq:main-estimate} for smooth data, with a constant independent of higher
derivatives of $F$.
\end{proof}

To treat general data, we smooth the density and pass to the limit in the preceding estimate.

\begin{proof}[Completion of the proof of Theorem~\ref{thm:main}]
Let $F\in C^\alpha(M)$ satisfy \eqref{eq:data}, and let $\varphi$ be the prescribed bounded
pluripotential solution with zero $d\mu_g$ mean. We divide the proof into four steps.

\paragraph{Step 1.}
For $0<\varepsilon\leq1/4$, set
\[
 f_\varepsilon=P_{\varepsilon^2}(e^F),\qquad F_\varepsilon=\log f_\varepsilon.
\]
We use the heat semigroup identities \eqref{pt a} and the estimates \eqref{eq:heat-alpha}
from Section~\ref{Heat estimates and interpolation}. The function $f_\varepsilon$ is smooth and positive, so $F_\varepsilon$ is smooth as well.
By positivity and preservation of mass,
\[
 e^{-K}\leq f_\varepsilon\leq e^K,
 \qquad \int_Mf_\varepsilon\,d\mu_g=\int_Me^F\,d\mu_g=1.
\]
Thus the smoothed densities satisfy the compatibility condition. Since $\|e^F\|_{C^\alpha}\leq e^K(1+K)$, the H\"older estimates imply
\[
 \|f_\varepsilon\|_{C^\alpha}\leq C,
 \qquad \|f_\varepsilon-e^F\|_\infty\leq C\varepsilon^\alpha.
\]
On $[e^{-K},e^K]$, the derivative of $\log$ is at most $e^K$. Hence, we have
\[
\begin{gathered}
 \|F_\varepsilon\|_\infty\leq K,
 \qquad [F_\varepsilon]_{C^\alpha}\leq e^K[f_\varepsilon]_{C^\alpha},\\
 \|F_\varepsilon-F\|_\infty
 \leq e^K\|f_\varepsilon-e^F\|_\infty\leq C\varepsilon^\alpha.
\end{gathered}
\]
In particular, there is a common bound
\[
 \|F_\varepsilon\|_{C^\alpha}\leq K_*,
 \qquad K_*=C(M,g,n,\alpha,K),
\]
where we may take $K_*\geq K$. We shall use this bound together with the uniform convergence
$F_\varepsilon\to F$.

\paragraph{Step 2.}
By Yau \cite{Yau}, there is a smooth normalized solution $\varphi_\varepsilon$ satisfying
\[
\begin{gathered}
 (\omega_g+\pbp\varphi_\varepsilon)^n=f_\varepsilon\omega_g^n,
 \qquad \omega_{\varphi_\varepsilon}>0,\\
 \int_M\varphi_\varepsilon\,d\mu_g=0.
\end{gathered}
\]
Applying the smooth estimate with $K_*$ in place of $K$, we obtain
\[
 \|\varphi_\varepsilon\|_{C^{2,\alpha}(M,g)}
 \leq C(M,g,n,\alpha,K_*)\leq C(M,g,n,\alpha,K).
\]
Here the dependence of $K_*$ on the original data has been absorbed into the final constant. The approximation estimate \eqref{eq:density-approximation} applies to the prescribed bounded
solution. In particular,
\[
 \|\varphi_\varepsilon-\varphi\|_\infty
 \leq C\varepsilon^{1+\alpha}(1+|\log\varepsilon|)\longrightarrow0.
\]
Thus the approximations converge uniformly to the prescribed pluripotential solution, which
has a continuous representative.

\paragraph{Step 3.}
Fix $0<\alpha'<\alpha$, for example $\alpha'=\alpha/2$, and choose a sequence
$\varepsilon_j\downarrow0$. By the uniform $C^{2,\alpha}$ bound and the Arzel\`a--Ascoli
theorem on the finite atlas, we can extract a subsequence converging in $C^2(M)$.
H\"older interpolation of the second derivatives then implies convergence in
$C^{2,\alpha'}(M)$. Retain the notation $\varepsilon_j$ for this subsequence. The uniform convergence in
Step~2 identifies the limit with $\varphi$. In particular, all first and second real
coordinate derivatives converge uniformly on the smaller charts.

For any pair of distinct points $x,y$ in one such chart, we may therefore pass to the limit in the
uniform H\"older estimate and obtain
\[
\begin{aligned}
 |D^2\varphi(x)-D^2\varphi(y)|
 &=\lim_{j\to\infty}
   |D^2\varphi_{\varepsilon_j}(x)-D^2\varphi_{\varepsilon_j}(y)|\\
 &\leq C|x-y|^\alpha.
\end{aligned}
\]
Here $D^2$ is the real coordinate Hessian. Since the constant is independent of the pair,
we can take the supremum of the H\"older quotients over $x\ne y$. This proves the full
$\alpha$-H\"older bound for $D^2\varphi$. The lower-order bounds also pass to the limit.
Combining the estimates on the finitely many charts, we obtain
\[
 \varphi\in C^{2,\alpha}(M),\qquad
 \|\varphi\|_{C^{2,\alpha}(M,g)}\leq C(M,g,n,\alpha,K).
\]
Thus compactness is used at the exponent $\alpha'$, and the uniform H\"older inequalities retain the
original exponent $\alpha$.

\paragraph{Step 4.}
It remains to verify strict positivity of the limiting metric and the classical equation.
By the uniform $C^{2,\alpha}$ estimate,
\[
 v_{\varphi_\varepsilon}=n+\Delta_g\varphi_\varepsilon\leq L,
\]
where $L\geq n$ depends only on $M,g,n,\alpha,K$. If $\lambda_1,\ldots,\lambda_n>0$ are the
eigenvalues of $g_{\varphi_\varepsilon}$ relative to $g$, then
\[
 \lambda_i\leq L,\qquad
 \prod_{j=1}^n\lambda_j=f_\varepsilon\geq e^{-K}.
\]
For each $i$, it follows that
\[
 \lambda_i=\frac{f_\varepsilon}{\prod_{j\ne i}\lambda_j}
           \geq\frac{e^{-K}}{L^{n-1}}.
\]
Consequently,
\[
 \frac{e^{-K}}{L^{n-1}}\,\omega_g
 \leq\omega_{\varphi_\varepsilon}\leq L\omega_g.
\]
Passing to the limit using the $C^2$ convergence in Step~3, we obtain
\[
 \frac{e^{-K}}{L^{n-1}}\,\omega_g\leq\omega_\varphi\leq L\omega_g.
\]
In particular, $\omega_\varphi>0$. Since the determinant is continuous in the matrix entries
and $f_\varepsilon\to e^F$ uniformly, we can pass to the limit pointwise in the smooth equations
along the subsequence from Step~3. Uniform convergence of the potentials also preserves the
mean-zero normalization. Thus the solution satisfies \eqref{eq:MA} classically.

Finally, suppose that the bounded solution is normalized by $\sup_M\varphi=0$. Set
\[
 \widetilde\varphi=\varphi-\int_M\varphi\,d\mu_g.
\]
Then $\widetilde\varphi$ has zero background mean and solves the same equation. By the
mean-zero estimate, $\|\widetilde\varphi\|_{C^{2,\alpha}}\leq C$. Since
\[
 \sup_M\widetilde\varphi=-\int_M\varphi\,d\mu_g,
 \qquad \varphi=\widetilde\varphi-\sup_M\widetilde\varphi,
\]
we have $\|\varphi\|_\infty\leq2\|\widetilde\varphi\|_\infty$. All derivatives and their H\"older
seminorms are unchanged by this constant shift. The same $C^{2,\alpha}$ estimate therefore holds
with the supremum normalization, after increasing its constant. This completes the proof.
\end{proof}

\section{Local $C^{2,\alpha}$ estimates}
\label{sec:local-obstructions}

In this section, we collect local estimates on a coordinate ball and compare them with the global estimate.
We first exhibit a counterexample and then state sufficient conditions for an interior estimate.
Let $B_r=\{z\in\mathbb C^n:|z|<r\}$, and let $g$ be a fixed smooth K\"ahler metric on a
neighborhood of $\overline{B_1}$. Assume that the bounded $\omega_g$-plurisubharmonic
potential $\varphi$ satisfies
\beq
 \omega_\varphi^n=e^F\omega_g^n\quad\hbox{in }B_1,\qquad
 \|F\|_{C^\alpha(B_1)}\leq K,\qquad
 \|\varphi\|_{L^\infty(B_1)}\leq A.
 \label{eq:local-data}
\eeq
The equation holds in the pluripotential sense. Set $d\mu_g=V_{B_1}^{-1}\omega_g^n$, where
$V_{B_1}=\int_{B_1}\omega_g^n$, so that $\int_{B_1}d\mu_g=1$. The local problem requires no
integral normalization of $e^F$. As before, $\omega_\varphi=\omega_g+\pbp\varphi$ and
$v_\varphi=n+\Delta_g\varphi$. For a weak potential, $v_\varphi$ is understood as a distribution.
Under \eqref{eq:local-data} alone, an estimate of the form
\beq
 \|\varphi\|_{C^{2,\alpha}(B_{1/2})}\leq C(n,\alpha,g,K,A)
 \label{eq:false-estimate}
\eeq
does not hold. By adding flat factors to the two-dimensional Pogorelov-type example discussed by
B{\l}ocki \cite{Blo} and He \cite[Section~3]{He}, we obtain a bounded Lipschitz potential with
density $1$ that is not differentiable at the origin. Smooth counterexamples can then be
constructed by approximating its boundary values and solving the Dirichlet problem with the
same density.

\begin{prop}
\label{prop:smooth-sequence}
(cf. B{\l}ocki \cite{Blo} and He \cite{He})
For every $n\geq2$ and $0<\alpha<1$, there are smooth K\"ahler potentials $\varphi_j$ on
$\overline{B_1}$ for the standard flat metric $\omega_g=\pbp|z|^2$ such that
$$
 \omega_{\varphi_j}^n=\omega_g^n,\qquad
 \sup_j\|\varphi_j\|_{L^\infty(B_1)}<\infty,
$$
but
$$
 \|\varphi_j\|_{C^{2,\alpha}(B_{1/2})}\longrightarrow\infty.
$$
In fact, $\|\varphi_j\|_{C^2(B_{1/2})}\longrightarrow\infty$.
\end{prop}

\begin{proof}
We use the two-dimensional example of B{\l}ocki \cite{Blo} with the normalization of
He \cite[Section~3]{He} and set
$$
 \Phi(z)=2(1+|z_1|^2)|z_2|+\sum_{j=3}^n|z_j|^2.
$$
The two-dimensional factor has density $1$, and the complex Hessian in the remaining
directions is the identity. Hence $\Phi\in\PSH(B_1)\cap C^{0,1}(\overline{B_1})$ satisfies
$$
 (\pbp\Phi)^n=\omega_g^n\quad\hbox{in }B_1
$$
in the pluripotential sense. Moreover, $\Phi$ is not differentiable at the origin.

Choose $h_j\in C^\infty(\partial B_1)$ such that
$\delta_j:=\|h_j-\Phi\|_{L^\infty(\partial B_1)}\to0$. By
Caffarelli--Kohn--Nirenberg--Spruck \cite{CKNS}, the Dirichlet problem
$$
 \det((\Psi_j)_{i\bar j})=1\quad\hbox{in }B_1,\qquad
 \Psi_j=h_j\quad\hbox{on }\partial B_1
$$
has a strictly plurisubharmonic solution $\Psi_j\in C^\infty(\overline{B_1})$. The functions
$\Phi\pm\delta_j$ have the same Monge--Amp\`ere measure as $\Psi_j$, and their boundary
values bound $h_j$ from below and above. Applying the comparison principle of
Bedford--Taylor \cite[Theorem~A]{BT1}, we obtain
$$
 \Phi-\delta_j\leq\Psi_j\leq\Phi+\delta_j
 \quad\hbox{on }\overline{B_1}.
$$
Set $\varphi_j=\Psi_j-|z|^2$. The potentials $\varphi_j$ are smooth and K\"ahler with
density $1$. They are uniformly bounded in $L^\infty$ and converge uniformly to
$\varphi=\Phi-|z|^2$.

Suppose a subsequence had bounded $C^2(B_{1/2})$ norms. By the Arzel\`a--Ascoli theorem, it
would have a further subsequence converging in $C^1(\overline{B_{1/4}})$. Its limit must be
$\varphi$ by uniform convergence. This contradicts the failure of differentiability at
the origin, since $\Phi$ differs from $\varphi$ by a smooth function. Hence the $C^2$ norms tend to infinity, and so do the $C^{2,\alpha}$ norms.
\end{proof}

\begin{rem}
The global estimate in Theorem~\ref{thm:main} relies on the uniform Green and Sobolev estimates
\eqref{eq:Green-estimate} and \eqref{eq:Sobolev-estimate} in the fixed K\"ahler class on a compact
manifold. In Section~\ref{sec:parabolic-energy}, we integrate the energy inequalities over all of $M$
without spatial cutoffs.
With a spatial cutoff $\chi$, the localized argument has terms such as $|\Na\chi|_\varphi^2$.
To estimate them, we need control of the varying metric beyond its determinant. The local bounds in \eqref{eq:local-data} alone do
not provide this control. Proposition~\ref{prop:smooth-sequence} shows that the corresponding
local estimate \eqref{eq:false-estimate} fails even for the standard flat metric and density $1$.
\end{rem}

We now establish local estimates under additional assumptions on the potential or its trace.
The first estimate follows from the theorem of Li--Li--Zhang. To apply it, add a fixed local
potential of the background form, so that the resulting plurisubharmonic function solves the
Euclidean complex Monge--Amp\`ere equation.

\begin{theo}
\label{thm:local-Holder}
(cf. Li--Li--Zhang \cite{LLZ}) Under \eqref{eq:local-data}, set
$$
 \beta_*(n,\alpha)=1-\frac{\alpha}{n(2+\alpha)-1}.
$$
If, in addition,
$$
 \varphi\in C^{1,\beta}(B_1),\qquad
 \|\varphi\|_{C^{1,\beta}(B_1)}\leq B,\qquad
 \beta_*(n,\alpha)<\beta<1,
$$
then
$$
 \|\varphi\|_{C^{2,\alpha}(B_{1/2})}\leq C(n,\alpha,\beta,g,K,B).
$$
More generally, fix $0<r<R\leq1$. If $\|\varphi\|_{C^{1,\beta}(B_R)}\leq B$ for the same range
of $\beta$, then we have
$$
 \|\varphi\|_{C^{2,\alpha}(B_r)}\leq C(n,\alpha,\beta,g,K,B,r,R).
$$
\end{theo}

\begin{proof}
Fix $0<r<R\leq1$. Choose a smooth real local potential $\rho$ for the background metric on
a neighborhood of $\overline{B_1}$, so that $\rho_{i\bar j}=g_{i\bar j}$. Set
$\Phi=\rho+\varphi$. Then
$\Phi\in\PSH(B_R)\cap C^{1,\beta}(B_R)$ and
$$
 \det(\Phi_{i\bar j})=a,\qquad a=e^F\det(g_{i\bar j}),
$$
where the equation holds in the pluripotential sense. The hypotheses imply
$$
\begin{gathered}
 a\geq e^{-K}\inf_{B_1}\det(g_{i\bar j})>0,
 \qquad \|a\|_{C^\alpha(B_R)}\leq C(g,\alpha,K),\\
 \|\Phi\|_{C^{1,\beta}(B_R)}\leq B+\|\rho\|_{C^{1,\beta}(B_R)}.
\end{gathered}
$$
Here the bound for $a$ follows from the H\"older product inequality and
$\|e^F\|_{C^\alpha}\leq e^K(1+K)$. All norms of the fixed function $\rho$ are included in the
background constants.

Since $\beta>\beta_*(n,\alpha)$, we can apply the theorem of Li--Li--Zhang
\cite[Theorem~1.2]{LLZ} to $\Phi$ on $B_R$. Their estimate on $B_r\Subset B_R$ has the asserted
dependence, including on $\operatorname{dist}(B_r,\partial B_R)=R-r$. After subtracting $\rho$,
we obtain the general statement. The first estimate is the special case $R=1$ and $r=1/2$.
The argument applies directly to bounded pluripotential solutions.
\end{proof}

He \cite{He} obtained local Laplacian estimates from an integral bound and a bound for the first
derivatives of the logarithmic density. The following sufficient condition for $C^\alpha$ data
follows from background elliptic regularity and Theorem~\ref{thm:local-Holder}.

\begin{theo}
\label{thm:local-Lp}
In addition to \eqref{eq:local-data}, assume that the distribution $v_\varphi=n+\Delta_g\varphi$ is
represented by an $L^p(B_1,d\mu_g)$ function satisfying
\beq
 \|v_\varphi\|_{L^p(B_1,d\mu_g)}\leq L,\qquad
 \frac{2n\bigl(n(2+\alpha)-1\bigr)}\alpha<p<\infty.
 \label{eq:Lp-threshold}
\eeq
Then
\beq
 \|\varphi\|_{C^{2,\alpha}(B_{1/2})}
 \leq C(n,\alpha,p,g,K,A,L).
 \label{eq:Lp-conclusion}
\eeq
The threshold in \eqref{eq:Lp-threshold} is a sufficient condition. Its optimality is not asserted.
\end{theo}

\begin{proof}
From the distributional equation $\Delta_g\varphi=v_\varphi-n$ and the interior estimate of
Gilbarg--Trudinger \cite[Chapter~9]{GilTru}, we obtain
$$
 \|\varphi\|_{W^{2,p}(B_{7/8})}\leq C(g,n,p)(L+A+1).
$$
To justify the estimate for a weak potential, first solve the same Poisson equation with
zero boundary values on $B_{15/16}$ in $W^{2,p}$. The difference between $\varphi$ and this
solution is distributionally harmonic for $\Delta_g$ and hence smooth in the interior.
Thus $\varphi$ has local $W^{2,p}$ regularity, and the stated interior estimate applies.

By \eqref{eq:Lp-threshold}, we may fix $\beta$ depending only on $n,\alpha,p$ such that
$$
 \beta_*(n,\alpha)<\beta<1-\frac{2n}{p}.
$$
The Sobolev embedding in real dimension $2n$ applies with this exponent. By
Gilbarg--Trudinger \cite[Chapter~7]{GilTru},
$$
 \|\varphi\|_{C^{1,\beta}(B_{3/4})}
 \leq C(g,n,p,\alpha)(L+A+1).
$$
The continuous representative agrees pointwise with the original potential. Indeed, after
adding a smooth local potential of $\omega_g$, the resulting plurisubharmonic function is
recovered pointwise by shrinking ball averages, and these averages depend only on its
almost-everywhere values. Applying Theorem~\ref{thm:local-Holder} with $R=3/4$ and $r=1/2$, we
obtain \eqref{eq:Lp-conclusion} with the asserted dependence of the constant.
\end{proof}

Finally, a uniform trace bound controls every finite $L^p$ norm on the ball. The local regularity
result under a bounded Laplacian is due to Wang \cite{WY}. See also the estimates under a bound for
the complex Hessian in Chen--Wang \cite{CW}. The following corollary is a consequence of
Theorem~\ref{thm:local-Lp}.

\begin{cor}
\label{cor:local-trace}
(cf. Wang \cite{WY})
Under \eqref{eq:local-data}, assume that $v_\varphi=n+\Delta_g\varphi$ is represented by an
$L^\infty(B_1)$ function with
$$
 \|v_\varphi\|_{L^\infty(B_1)}\leq L.
$$
Then
$$
 \|\varphi\|_{C^{2,\alpha}(B_{1/2})}\leq C(n,\alpha,g,K,A,L).
$$
\end{cor}

\begin{proof}
Choose $p=4n\bigl(n(2+\alpha)-1\bigr)/\alpha$, which satisfies \eqref{eq:Lp-threshold} and depends
only on $n,\alpha$. Since $\int_{B_1}d\mu_g=1$, we have $\|v_\varphi\|_{L^p(B_1,d\mu_g)}\leq L$. The
assertion follows from Theorem~\ref{thm:local-Lp}.
\end{proof}

\vskip10pt
Haozhao Li, Institute of Geometry and Physics, and School of Mathematical
Sciences, University of Science and Technology of China, No. 96 Jinzhai
Road, Hefei, Anhui Province, 230026, China.
Email: \href{mailto:hzli@ustc.edu.cn}{hzli@ustc.edu.cn}\par\medskip

Bing Wang, Institute of Geometry and Physics, and School of Mathematical
Sciences, University of Science and Technology of China, Hefei 230026,
China, and Hefei National Laboratory, Hefei 230088, China.
Email: \href{mailto:topspin@ustc.edu.cn}{topspin@ustc.edu.cn}\par\medskip

Kai Zheng, University of Chinese Academy of Sciences, Beijing, 100190, China.
Email: \href{mailto:KaiZheng@amss.ac.cn}{KaiZheng@amss.ac.cn}\par\medskip


\begin{thebibliography}{99}

\bibitem{Au}
T. Aubin, \emph{\'Equations du type Monge--Amp\`ere sur les vari\'et\'es k\"ahl\'eriennes compactes}, Bull. Sci. Math. (2) \textbf{102} (1978), no.~1, 63--95.

\bibitem{BT1}
E. Bedford and B. A. Taylor, \emph{The Dirichlet problem for a complex Monge--Amp\`ere equation}, Invent. Math. \textbf{37} (1976), no.~1, 1--44.

\bibitem{BT2}
E. Bedford and B. A. Taylor, \emph{A new capacity for plurisubharmonic functions}, Acta Math. \textbf{149} (1982), nos.~1--2, 1--40.

\bibitem{Blo}
Z. B{\l}ocki, \emph{On the regularity of the complex Monge--Amp\`ere operator}, in \emph{Complex Geometric Analysis in Pohang}, K.-T. Kim and S. G. Krantz (eds.), Contemp. Math. \textbf{222}, Amer. Math. Soc., Providence, RI, 1999, 181--189.

\bibitem{Blocki2000}
Z. B{\l}ocki, \emph{Interior regularity of the complex Monge--Amp\`ere equation in convex domains}, Duke Math. J. \textbf{105} (2000), no.~1, 167--181.

\bibitem{BlockiGradient}
Z. B{\l}ocki, \emph{A gradient estimate in the Calabi--Yau theorem}, Math. Ann. \textbf{344} (2009), no.~2, 317--327.

\bibitem{Blocki1}
Z. B{\l}ocki, \emph{On uniform estimate in Calabi--Yau theorem}, Sci. China Ser. A \textbf{48} (2005), suppl., 244--247.

\bibitem{Blocki2}
Z. B{\l}ocki, \emph{On the uniform estimate in the Calabi--Yau theorem, II}, Sci. China Math. \textbf{54} (2011), no.~7, 1375--1377.

\bibitem{BD}
Z. B{\l}ocki and S. Dinew, \emph{A local regularity of the complex Monge--Amp\`ere equation}, Math. Ann. \textbf{351} (2011), no.~2, 411--416.

\bibitem{BEGZ2010}
S. Boucksom, P. Eyssidieux, V. Guedj and A. Zeriahi, \emph{Monge--Amp\`ere equations in big cohomology classes}, Acta Math. \textbf{205} (2010), no.~2, 199--262.

\bibitem{Caffarelli1990}
L. A. Caffarelli, \emph{Interior $W^{2,p}$ estimates for solutions of the Monge--Amp\`ere equation}, Ann. of Math. (2) \textbf{131} (1990), no.~1, 135--150.

\bibitem{Caffarelli1991}
L. A. Caffarelli, \emph{Some regularity properties of solutions of Monge Amp\`ere equation}, Comm. Pure Appl. Math. \textbf{44} (1991), nos.~8--9, 965--969.

\bibitem{CKNS}
L. Caffarelli, J. J. Kohn, L. Nirenberg and J. Spruck, \emph{The Dirichlet problem for nonlinear second-order elliptic equations. II. Complex Monge--Amp\`ere, and uniformly elliptic, equations}, Comm. Pure Appl. Math. \textbf{38} (1985), no.~2, 209--252.

\bibitem{ChenChengCscK1}
X. X. Chen and J. R. Cheng, \emph{On the constant scalar curvature K\"ahler metrics (I)---A priori estimates}, J. Amer. Math. Soc. \textbf{34} (2021), no.~4, 909--936.

\bibitem{CC2}
X. X. Chen and J. R. Cheng, \emph{Some new estimates for the complex Monge--Amp\`ere equation}, Sci. China Math. \textbf{62} (2019), no.~11, 2073--2088.

\bibitem{CC}
X. X. Chen and J. R. Cheng, \emph{Gradient estimate for complex Monge--Amp\`ere equation with continuous right hand side}, arXiv:2201.13330 (2022).

\bibitem{MR2348984}
X. X. Chen and W. Y. Ding, \emph{Ricci flow on surfaces with degenerate initial metrics}, J. Partial Differential Equations \textbf{20} (2007), no.~3, 193--202.

\bibitem{CH}
X. X. Chen and W. Y. He, \emph{The complex Monge--Amp\`ere equation on compact K\"ahler manifolds}, Math. Ann. \textbf{354} (2012), no.~4, 1583--1600.

\bibitem{MR2775789}
X. X. Chen, G. Tian and Z. Zhang, \emph{On the weak K\"ahler--Ricci flow}, Trans. Amer. Math. Soc. \textbf{363} (2011), no.~6, 2849--2863.

\bibitem{CW}
X. X. Chen and Y. Q. Wang, \emph{$C^{2,\alpha}$-estimate for Monge--Amp\`ere equations with H\"older-continuous right hand side}, Ann. Global Anal. Geom. \textbf{49} (2016), no.~2, 195--204.

\bibitem{Chu}
J. Chu, \emph{$C^{2,\alpha}$ regularities and estimates for nonlinear elliptic and parabolic equations in geometry}, Calc. Var. Partial Differential Equations \textbf{55} (2016), no.~1, Paper No.~8, 20 pp.

\bibitem{Davies}
E. B. Davies, \emph{Heat Kernels and Spectral Theory}, Cambridge Tracts in Mathematics, vol.~92, Cambridge Univ. Press, Cambridge, 1989.

\bibitem{DDGHKZ}
J.-P. Demailly, S. Dinew, V. Guedj, Pham Hoang Hiep, S. Ko{\l}odziej and A. Zeriahi, \emph{H\"older continuous solutions to Monge--Amp\`ere equations}, J. Eur. Math. Soc. \textbf{16} (2014), no.~4, 619--647.

\bibitem{DZZ}
S. Dinew, X. Zhang and X. W. Zhang, \emph{The $C^{2,\alpha}$ estimate of complex Monge--Amp\`ere equation}, Indiana Univ. Math. J. \textbf{60} (2011), no.~5, 1713--1722.

\bibitem{Evans1982}
L. C. Evans, \emph{Classical solutions of fully nonlinear, convex, second-order elliptic equations}, Comm. Pure Appl. Math. \textbf{35} (1982), no.~3, 333--363.

\bibitem{EGZ2009}
P. Eyssidieux, V. Guedj and A. Zeriahi, \emph{Singular K\"ahler--Einstein metrics}, J. Amer. Math. Soc. \textbf{22} (2009), no.~3, 607--639.

\bibitem{GilTru}
D. Gilbarg and N. S. Trudinger, \emph{Elliptic Partial Differential Equations of Second Order}, Classics in Mathematics, Springer-Verlag, Berlin, 2001.

\bibitem{Grig}
A. Grigor'yan, \emph{Heat Kernel and Analysis on Manifolds}, AMS/IP Studies in Advanced Mathematics, vol.~47, Amer. Math. Soc., Providence, RI, and International Press, Boston, MA, 2009.

\bibitem{GT}
V. Guedj and T. D. T\^o, \emph{K\"ahler families of Green's functions}, J. \'Ec. polytech. Math. \textbf{12} (2025), 319--339.

\bibitem{MR3680377}
V. Guedj and A. Zeriahi
\emph{Regularizing properties of the twisted K\"ahler-Ricci flow},
J. Reine Angew. Math. 729 (2017), 275--304.


\bibitem{GPSSs}
B. Guo, D. H. Phong, J. Song and J. Sturm, \emph{Sobolev inequalities on K\"ahler spaces}, arXiv:2311.00221 (2023).

\bibitem{GPSSd}
B. Guo, D. H. Phong, J. Song and J. Sturm, \emph{Diameter estimates in K\"ahler geometry}, Comm. Pure Appl. Math. \textbf{77} (2024), no.~8, 3520--3556.

\bibitem{GPSSs2}
B. Guo, D. H. Phong, J. Song and J. Sturm, \emph{Diameter estimates in K\"ahler geometry II: removing the small degeneracy assumption}, Math. Z. \textbf{308} (2024), no.~3, Paper No.~43, 7 pp.

\bibitem{GPS}
B. Guo, D. H. Phong and J. Sturm, \emph{Green's functions and complex Monge--Amp\`ere equations}, J. Differential Geom. \textbf{127} (2024), no.~3, 1083--1119.

\bibitem{GuoPhongTong2023}
B. Guo, D. H. Phong and F. Tong, \emph{On $L^\infty$ estimates for complex Monge--Amp\`ere equations}, Ann. of Math. (2) \textbf{198} (2023), no.~1, 393--418.

\bibitem{GuoPhongTongGradient}
B. Guo, D. H. Phong and F. Tong, \emph{A new gradient estimate for the complex Monge--Amp\`ere equation}, Math. Ann. \textbf{388} (2024), no.~1, 1045--1051.

\bibitem{He}
W. He, \emph{On the regularity of the complex Monge--Amp\`ere equations}, Proc. Amer. Math. Soc. \textbf{140} (2012), no.~5, 1719--1727.

\bibitem{Ji}
Q. Ji, \emph{On the $C^0$-estimate in Calabi--Yau theorem}, Nonlinear Anal. \textbf{64} (2006), no.~11, 2492--2495.

\bibitem{Kol}
S. Ko{\l}odziej, \emph{The complex Monge--Amp\`ere equation}, Acta Math. \textbf{180} (1998), no.~1, 69--117.

\bibitem{Kol2008}
S. Ko{\l}odziej, \emph{H\"older continuity of solutions to the complex Monge--Amp\`ere equation with the right-hand side in $L^p$: the case of compact K\"ahler manifolds}, Math. Ann. \textbf{342} (2008), no.~2, 379--386.

\bibitem{Krylov1982}
N. V. Krylov, \emph{Boundedly nonhomogeneous elliptic and parabolic equations}, Izv. Akad. Nauk SSSR Ser. Mat. \textbf{46} (1982), no.~3, 487--523. English translation in Math. USSR-Izv. \textbf{20} (1983), no.~3, 459--492.

\bibitem{Ledoux}
M. Ledoux, \emph{A simple analytic proof of an inequality by P. Buser}, Proc. Amer. Math. Soc. \textbf{121} (1994), no.~3, 951--959.

\bibitem{LLZ}
C. Li, J. Li and X. Zhang, \emph{A $C^{2,\alpha}$ estimate of the complex Monge--Amp\`ere equation}, J. Funct. Anal. \textbf{275} (2018), no.~1, 149--169.

\bibitem{Lunardi}
A. Lunardi, \emph{Analytic Semigroups and Optimal Regularity in Parabolic Problems}, Progress in Nonlinear Differential Equations and their Applications, vol.~16, Birkh\"auser Verlag, Basel, 1995.

\bibitem{MR3570152}
E. Di Nezza and C. H. Lu
\emph{Uniqueness and short time regularity of the weak K\"ahler-Ricci flow},
Adv. Math. 305 (2017), 953--993.

\bibitem{Nirenberg1953}
L. Nirenberg, \emph{On nonlinear elliptic partial differential equations and H\"older continuity}, Comm. Pure Appl. Math. \textbf{6} (1953), no.~1, 103--156.

\bibitem{PhongSturmGradient}
D. H. Phong and J. Sturm, \emph{On pointwise gradient estimates for the complex Monge--Amp\`ere equation}, in \emph{Advances in Geometric Analysis}, Adv. Lect. Math. (ALM), vol.~21, International Press, Somerville, MA, 2012, 87--95.

\bibitem{Schauder1934}
J. Schauder, \emph{\"Uber lineare elliptische Differentialgleichungen zweiter Ordnung}, Math. Z. \textbf{38} (1934), 257--282.

\bibitem{MR3595934}
J. Song and G. Tian
\emph{The K\"ahler-Ricci flow through singularities},
Invent. Math. 207 (2017), no. 2, 519--595.


\bibitem{ST}
G. Sz\'ekelyhidi and V. Tosatti, \emph{Regularity of weak solutions of a complex Monge--Amp\`ere equation}, Anal. PDE \textbf{4} (2011), no.~3, 369--378.

\bibitem{TosattiWeinkove2010}
V. Tosatti and B. Weinkove, \emph{The complex Monge--Amp\`ere equation on compact Hermitian manifolds}, J. Amer. Math. Soc. \textbf{23} (2010), no.~4, 1187--1195.

\bibitem{TWWY}
V. Tosatti, Y. Wang, B. Weinkove and X. Yang, \emph{$C^{2,\alpha}$ estimates for nonlinear elliptic equations in complex and almost complex geometry}, Calc. Var. Partial Differential Equations \textbf{54} (2015), no.~1, 431--453.

\bibitem{TrudingerWang2008}
N. S. Trudinger and X.-J. Wang, \emph{Boundary regularity for the Monge--Amp\`ere and affine maximal surface equations}, Ann. of Math. (2) \textbf{167} (2008), no.~3, 993--1028.

\bibitem{WY}
Y. Wang, \emph{On the $C^{2,\alpha}$-regularity of the complex Monge--Amp\`ere equation}, Math. Res. Lett. \textbf{19} (2012), no.~4, 939--946.

\bibitem{Yau}
S.-T. Yau, \emph{On the Ricci curvature of a compact K\"ahler manifold and the complex Monge--Amp\`ere equation. I}, Comm. Pure Appl. Math. \textbf{31} (1978), no.~3, 339--411.

\end{thebibliography}
\end{document}